\documentclass[11pt]{amsart}

\usepackage{color, verbatim}
\usepackage[T1]{fontenc}

\usepackage{amsmath,amsthm,amssymb,epsfig,fullpage}
\usepackage{latexsym}
\usepackage[T1]{fontenc} 

\graphicspath{{figures}}

\newtheorem{theorem}{Theorem}

\newtheorem{lemma}{Lemma}
\newtheorem{remark}{Remark}

\newtheorem{definition}{Definition}

\newtheorem{assumption}{Assumption}
\numberwithin{theorem}{section}
\numberwithin{lemma}{section}
\numberwithin{proposition}{section}
\numberwithin{equation}{section}
\numberwithin{remark}{section}
\numberwithin{remarks}{section}
\numberwithin{definition}{section}
\numberwithin{example}{section}
\numberwithin{corollary}{section}

\newcommand{\bz}{{\bf z}}

\newcommand{\rev}[1]{#1}

\def\R{\mathbb{R}}

\def\L2{{L^2(\O)}}

\newcommand {\D} {\displaystyle}

\def\N{\mathcal{N}_h}
\def\T{\mathcal{T}}
\def\S{\mathbb{S}}

\def\O{\Omega}

\def\x{{\bf x}}

\begin{document}
\title[Numerics for bi-harmonic wave maps into spheres]
{Numerical approximation of bi-harmonic wave maps into spheres}

\author{\v{L}ubom\'{i}r Ba\v{n}as}
\address{Department of Mathematics, Bielefeld University, 33501 Bielefeld, Germany}
\email{banas@math.uni-bielefeld.de}

\author{Sebastian Herr}
\address{Department of Mathematics, Bielefeld University, 33501 Bielefeld, Germany}
\email{herr@math.uni-bielefeld.de}

\date{\today}


\begin{abstract} 
We construct a structure preserving non-conforming finite element approximation scheme for the bi-harmonic wave maps into spheres equation.
It satisfies a discrete energy law and preserves the non-convex sphere constraint of the continuous problem.
The discrete sphere constraint is enforced at the mesh-points via a discrete Lagrange multiplier. This approach
restricts the spatial approximation to the (non-conforming) linear finite elements.
We show that the numerical approximation converges to the weak solution of the continuous problem in spatial dimension $d=1$.
The convergence analysis
in dimensions $d>1$ is  complicated by the lack of a discrete product rule as well as the low regularity of the numerical approximation in the non-conforming setting.
Hence, we show convergence of the numerical approximation in higher-dimensions by introducing additional stabilization terms in the numerical approximation.
We present numerical experiments to demonstrate the performance of the proposed numerical approximation
and to illustrate the regularizing effect of the bi-Laplacian which prevents the formation of singularities.
\end{abstract}


\thanks{Funded by the Deutsche Forschungsgemeinschaft (DFG, German Research Foundation) -- Project-ID 317210226 -- SFB 1283}
\maketitle
\section{Introduction}\label{sec_intro}


Let $T>0$, $\Omega \subset \R^d$ be an open, bounded, (convex) polyhedral domain, $d \leq 3$, and $ \Omega_T:= (0,T)\times \Omega$. Further,  let $\S^2 \subset \R^3$ denote the unit sphere.
We consider bi-harmonic wave maps $\mathbf{u}:(0,T)\times \Omega\to \S^2$. Formally, these are critical points of the action functional
$$
\Phi(\mathbf{u}):=\frac12 \int_{\Omega_T} |\partial_t\mathbf{u}|^2-|\Delta \mathbf{u}|^2 \mathrm{d}(t,x)
$$
under the sphere target constraint.
More precisely, we consider the Cauchy problem defined by the Euler-Lagrange equation, i.e.\
\begin{align}\label{biwave}
\partial_t^2 \mathbf{u} + \Delta^2 \mathbf{u} & \perp T_{\textbf{u}}(\S^2) \qquad  &\mbox{in } & \Omega_T,\qquad\qquad
\\ \label{wave1b}
\frac{\partial \mathbf{u}}{\partial {\bf n}} & = \frac{\partial \Delta \mathbf{u}}{\partial {\bf n}} = 0 \qquad
 & \mbox{on } &\partial \Omega,
\\ \label{wave1c}
\mathbf{u}(0) & = \mathbf{u}_0, \quad
\partial_t \mathbf{u}(0) = \mathbf{v}_0   \quad & \mbox{in } & \Omega,
\end{align}
where ${\bf n}$ is the outward unit normal vector field to $\partial \Omega$ and $T_{\textbf{u}(t,x)}(\S^2) $ denotes the tangent space to $\mathbf{u}(t,x)\in \S^2$.
\eqref{biwave} can be equivalently written as
\begin{equation}\label{biwave1}
\partial_t^2 \mathbf{u} + \Delta^2 \mathbf{u}=\lambda_w u, \; \text{with}\;
\lambda_w = \vert \Delta \mathbf{u}\vert^2 - \vert \partial_t \mathbf{u}\vert^2 - \Delta |\nabla \mathbf{u}| - 2 \mathrm{div} \langle \Delta \mathbf{u}, \nabla \mathbf{u}\rangle.
\end{equation}
There are a number of analytical results on this problem in the  setting $\Omega=\R^d$ without boundary condition.
In \cite{biharm}, weak solutions in the energy space are constructed by a suitable Ginzburg-Landau-type approximation. Furthermore, in \cite{hlss20} local well-posedness for initial data of sufficiently high Sobolev regularity has been proven (for more general targets), including a blow-up criterion. In dimension $d=1,2$, the energy conservation has been used to prove global existence \cite{s23}.

As far as we are aware the present paper is the first one addressing the numerical approximation of the bi-harmonic wave maps into spheres equation.

There are a number of results on the numerical analysis for (second order) wave maps, for instance in \cite{bfp07,num_harm09,kw14,bar15} for the sphere target
and with more general target manifolds in \cite{bar09}, \cite{num_move}.
The numerical approximation of wave maps shares many features with the numerical
approximation of its parabolic counterparts, the harmonic heat flow equation and with the Landau-Lifshitz-Gilbert equation. Concerning the numerical analysis and further references of these well-studied problems we refer to \cite{alo08,book_ferro,bar16}. 

\rev{In contrast to the aforementioned second order problems, which employ a $H^1$-conforming finite element spatial discretization,}
a conforming finite element discretization of problems involving (fourth order) bi-harmonic operators requires at least $C^1$ regularity of the employed polynomial approximation spaces, cf. \cite{book_brenner}.
Apart from the fact that the $H^2$-conforming finite element methods are rather complicated to implement, cf.\ e.g.\ \cite{p22}, the requirement to satisfy the sphere constraint
on the discrete level restricts the spatial discretization to a piecewise linear ($H^1$-conforming) finite element setting.
Non-conforming methods for fourth order problems based on piecewise linear finite elements have been studied in  \cite{be92,eymard11}.

Here, in order to show convergence of the numerical approximation we require strong convergence of the gradient of the numerical approximation,
which is complicated by the fact the the discrete solution is only $H^1$-regular. This prohibits the application of the standard Aubin-Lions-Simon compactness result for the compact embedding $H^2\subset H^1$. We derive a new discrete compactness result for the non-conforming setting.

The paper is organized as follows. Notation and preliminary results are introduced in Section~\ref{sec_not}.
The numerical scheme is introduced in Section~\ref{sec_num} and its convergence towards a weak solution for $d=1$ is
shown in Section~\ref{sec_conv}. In Section~\ref{sec_stab} we discuss a stabilized variant of the numerical scheme that is convergent in higher spatial dimensions.
In Section \ref{sec_conv} we present a variant of the numerical scheme which preserves the discrete energy.
Computational experiments are given in Section~\ref{sec_exp}.

\section{Notation and preliminaries}\label{sec_not}

We employ the following notation throughout this paper.
By $\mathbf{L}^2$, $\mathbf{L}^\infty$, $\mathbf{H}^1$ we denote the $\mathbb{R}^3$-valued function spaces
$L^2(\Omega, \mathbb{R}^3)$, $L^\infty(\Omega, \mathbb{R}^3)$, $H^{1}(\Omega, \mathbb{R}^3)$, respectively.
The standard norm in Lebesgue space $L^p(\Omega)$, $p>0$ is denoted as $\|\cdot\|_{L^p}$. For $p=2$ we denote the inner product on $L^2(\Omega)$ as $(\cdot,\cdot) := (\cdot,\cdot)_{L^2(\Omega)}$
and the corresponding norm by $\|\cdot\|:=\|\cdot\|_{L^2}$; the inner product on  $L^2(\omega)$ for $\omega\subset \Omega$ is denoted as $(\cdot,\cdot)_\omega := (\cdot,\cdot)_{L^2(\omega)}$.

Let $V_h\equiv V_h(\mathcal{T}_h) \subset H^{1}(\Omega)$ be the lowest order $H^1$-conforming finite element space subordinate to a simplicial partition
$\T_h$ of $\Omega$, and ${\bf V}_h := [V_h]^3$. By $\N$, we denote the set of all nodes, i.e., vertices of all elements $T \in \T_h$. Furthermore,
we consider the nodal basis $\{\varphi_\bz;\, \bz \in \N\}$ of $V_h$ consisting of piecewise linear continuous functions that satisfy $\varphi_{\bz_k}(\bz_l) = \delta_{kl}$ for $\bz_k, \bz_l\in \mathcal{N}_h$.
For $\bz \in \mathcal{N}_h$ we denote $\omega_\bz = \mathrm{supp}(\varphi_\bz)$. For $T\in \mathcal{T}_h$ we denote $h_T = \mathrm{diam}(T)$
and the mesh size is denoted as $h=\max_{T\in\mathcal{T}}h_T$.

We define the nodal interpolation operator $\mathcal{I}_h: C(\overline{\Omega}) \rightarrow V_h$ by
${\mathcal I}_h w := \sum_{\bz \in \N} w(\bz) \varphi_{\bz}$, for all $w \in C(\overline{\Omega})$.
The following interpolation estimate holds for $p\in (1, \infty]$, \cite{book_brenner}:
\begin{equation}\label{est_interpol}
\|w - {\mathcal I}_h w \|_{L^p} + h \|\nabla(w - {\mathcal I}_h w) \|_{L^p} \leq C h^2 \|D^2 w\|_{L^p} \quad \text{for } w\in W^{2,p}(\Omega),
\end{equation}
where $h = \max_{T\in \T_h} \mathrm{diam}(T)$.
We also note that the interpolation operator is $W^{1,\infty}$-stable, cf. e.g. \cite[Example 3.7]{book_bartels_fem}, i.e.:
$\|\mathcal{I}_h w\|_{W^{1,\infty}} \leq C \|w\|_{W^{1,\infty}}$ for $w \in W^{1,\infty}$.

We denote the inner product on $\mathbb{R}^3$ by $\langle \cdot, \cdot \rangle$
and define the discrete version of the $L^2$-inner product on ${\bf V}_h$ as
$$
(\mathbf{v}, \mathbf{w})_h := \int_{\Omega} {\mathcal I}_h \langle \mathbf{v}(x), \mathbf{w}(x)\rangle \,
{\rm d}{\bf x} = \sum_{\bz \in \N}  \beta_{\bz}\langle \mathbf{v}(\bz), \mathbf{w}(\bz)\rangle \qquad \forall \, \mathbf{v},\mathbf{w} \in [C(\overline{\Omega})]^3,
$$
where $\beta_{\bz} = \frac{|\omega_\bz|}{d+1} = \int_{\Omega} \varphi_{\bz}\, {\rm d}{\bf x}$, for all $\bz \in \N$.

It is well known the discrete inner product satisfies
\begin{align}\label{convdiscnorm}
|(\mathbf{v}, \mathbf{w})_h - (\mathbf{v}, \mathbf{w})| & \leq Ch \|\mathbf{v}\|\|\nabla \mathbf{w}\|\,,
\end{align}
and the corresponding norm  $\|\mathbf{v}\|_h^2 := (\mathbf{v}, \mathbf{v})_h$ is equivalent to the $L^2$-norm, i.e.,
\begin{align}\label{normeq}
 \|\mathbf{v}\|_h^2 \leq  \|\mathbf{v}\|^2 & \leq  (d+2) \|\mathbf{v}\|_h^2\,,
\end{align}
for $\mathbf{v},\mathbf{w}\in {\bf V}_h$.

We partition the time interval $[0,T]$ into equidistant subintervals $(t_{n-1}, t_n)$, $n=1,\dots,N$, $t_n = n \tau$  with a time-step size $\tau = \frac{T}{N} > 0$.
For a sequence $\{\psi^n \}_{n \geq 0}$ we denote 
$d_t \psi^n := \tau^{-1} \{ \psi^n - \psi^{n-1}\}$, and $\psi^{n+1/2} := \frac{1}{2} \{ \psi^{n+1} + \psi^n\}$
with $n \geq 1$.
By $C, \tilde{C}> 0$ we denote generic bounded constants which may depend on $\Omega$ and $T$ but are independent of $\tau$, $h$.

For $u \in H^{1}(\Omega)$ we define the discrete Laplacian $\Delta_h : H^{1}(\Omega) \rightarrow V_h$ as
\begin{equation}\label{disc_lap}
-(\Delta_h u, \varphi_h)_h = (\nabla u, \nabla \varphi_h)\qquad \forall \varphi_h \in V_h\,.
\end{equation}
The discrete Laplacian satisfies an inverse estimate
\begin{equation}\label{inv_disclap}
\|\Delta_h \phi_h\|_h^2 \leq \frac{1}{h^2}\|\nabla \phi_h\|^2\qquad \forall \phi_h \in V_h.
\end{equation}

To show convergence of the numerical apprpximation we require the following conditions to be satisfied.
\begin{assumption}\label{ass1}
We assume that the partition $\mathcal{T}_h$ satisfies:
\begin{itemize}
\item $\mathcal{T}_h$ is quasiuniform and consists of right-angled simplices, i.e. that, one of the following holds:
each simplex $T\in \mathcal{T}_h$ has one vertex such that all edges intersect at this vertex at right angles (Type-$1$ triangulation, $d=2,3$),
each simplex $T\in \mathcal{T}_h$ has two vertices at which two edges intersect at right angles (Type-$2$ triangulation, $d=3$).
\item
$\mathcal{T}_h$ is consistent with the discrete Laplace operator (\ref{disc_lap}) in the sense that
that for all $\pmb{\Phi} \in C^{\infty}_0 \bigl( [0,T); C^{\infty}_0(\Omega, \R^3)\bigr)$,
the discrete Laplacian $\Delta_h : C(\overline{\Omega}) \rightarrow V_h(\mathcal{T}_h)$
converges strongly for $h\rightarrow 0$, i.e.,
$$
\|\Delta_h\mathcal{I}_h\pmb{\Phi} - \Delta \pmb{\Phi} \|_h \rightarrow 0 \quad \mathrm{for}\quad h\rightarrow 0.
$$
\end{itemize}
\end{assumption}

\begin{remark}
The conditions of Assumption~\ref{ass1} can be verified explicitly for certain uniform partitions $\mathcal{T}_h$ of $\Omega$
on which the discrete Laplacian agrees with the finite difference approximation of the Laplace operator.
For $d=2$ the triangulation consisting of a union of halved squares with side $h$ (Type-$1$ triangulation)
agrees with the $5$-point finite difference approximation and for $d=3$
a partition that consists of cubes with side $h$ divided into six tetrahedra (Type-$2$ triangulation) agrees with the $7$-point finite difference
approximation.
In these two cases the standard theory of finite difference approximation implies that that
$$
\Delta_h\phi({\bf z}) - \Delta\phi ({\bf z}) = \mathcal{O}(h^2) \qquad \forall \phi \in C^4(\overline{\Omega})
$$
for all ${\bf z}\in \mathcal{N}_h$.
We note that on general meshes only weak convergence of the discrete Laplace operator can be expected, cf. \cite{eymard11}.
\end{remark}

\subsection{Weak solution}\label{sec_weak}
\begin{definition}\label{def_weak}
Let $(\, \mathbf{u}_0, {\bf v}_0\, ) \in H^{2}(\Omega, \R^3) \times L^2(\Omega, \R^3)$ be such that
$\vert \mathbf{u}_0\vert = 1$, and $\langle \mathbf{u}_0, {\bf v}_0\rangle 
= 0$  a.e.~in $\Omega \subset \R^d$.
We call a function $\mathbf{u}: \Omega_T \rightarrow \R^3$, such that
$\partial_t \mathbf{u}, \Delta \mathbf{u} \in L^2(\Omega_T, \R^3)$ a {\em weak solution} of (\ref{biwave}), if there holds:
\begin{enumerate}
\item[(i)] $\mathbf{u}(0) = \mathbf{u}_0$, and
      $\partial_t \mathbf{u}(0) = {\bf v}_0$ a.e. in $\Omega$,
\item[(ii)] $\vert \mathbf{u}\vert = 1$ a.e. in $\Omega_T$,
\item[(iii)] for all $\pmb{\Phi} \in C^{\infty}_0 \bigl( [0,T); W^{2,2}(\Omega, \R^3)\bigr)$ there holds
\begin{align}\label{weakform}
-\int_0^T \bigl( \partial_t \mathbf{u},\partial_t  (\mathbf{u} \times \pmb{\Phi})\bigr)\, {\rm d}t + 
\int_0^T \bigl( \Delta \mathbf{u} , \Delta (\mathbf{u}\times\pmb{\Phi})\bigr)\, {\rm d}t
= \bigl( {\bf v}_0 \wedge \mathbf{u}_0, \pmb{\Phi}(0)\bigr).
\end{align}
\end{enumerate}
\end{definition}
Recently, (global) existence of weak solutions ${\bf u}: [0,T]\times \Omega \rightarrow \mathbb{S}^2$ (for $\Omega \equiv \mathbb{R}^d$) has been established in \cite{biharm}.
Note that for $\mathbf{u},\pmb{\Phi}$ as in Definition \ref{def_weak} we have $ \bigl( \partial_t \mathbf{u},\partial_t  (\mathbf{u} \times \pmb{\Phi})\bigr)= \bigl( \partial_t \mathbf{u},\mathbf{u} \times \partial_t  \pmb{\Phi}\bigr)$ and that  $|\nabla u|^2\leq |\Delta u|$ a.e..

We remark that our definition of a weak solution differs slightly from \cite[(1.4)]{biharm} in the sense that in Definition \ref{def_weak} we restrict to a (dense) class of more regular test functions, but this is inessential. Similar to \cite[Lemma 2.1]{biharm} we observe that any sufficiently regular mapping $\mathbf{u}: \Omega_T \rightarrow \mathbb{R}^{m}$ satisfying
\begin{equation*}
\int_{\Omega_T} \Bigl[ -\langle \partial_t \mathbf{u}, \partial_t \pmb{\Phi}\rangle +
\langle\Delta \mathbf{u}, \Delta \pmb{\Phi}\rangle
- 
\big(\vert \Delta \mathbf{u}\vert^2 - \vert \partial_t \mathbf{u}\vert^2\big)\langle \mathbf{u}, \pmb{\Phi} \rangle
+ |\nabla \mathbf{u}|^2\Delta \langle \mathbf{u},\pmb{\Phi}\rangle
- 2\langle \Delta \mathbf{u} , \nabla \mathbf{u}\rangle \nabla \langle \mathbf{u},\pmb{\Phi}\rangle
\Bigr]
\, {\rm d}{\bf x}\,{\rm d}{t} = 0\,,
\end{equation*}
for all $\pmb{\Phi} \in C^{\infty}_0 \bigl( (0,T); W^{2,2}(\Omega, \R^3)\bigr)$ is a weak solution to \eqref{biwave}.

\section{Numerical approximation}\label{sec_num}
Let $\mathbf{U}^{0}$, $\mathbf{V}^{0}$, s.t. $\langle \mathbf{U}^{0}({\bf z}), \mathbf{V}^{0}({\bf z})\rangle = 0$ and $|{\bf U}({\bf z})| = 1$ for ${\bf z}\in\mathcal{N}_h$,
be given and set $\mathbf{U}^{-1} = \mathbf{U}^{0} - \tau \mathbf{V}^0$.
The numerical approximation of (\ref{biwave}) is then defined as follows: for $n=0,\dots, N-1$ determine $\mathbf{U}^{n+1}, \mathbf{W}^{n+1}\in {\bf V}_h$
such that for all $\pmb{\Phi}, \pmb{\Psi} \in {\bf V}_h$ there holds
\begin{align}\label{scheme}
\bigl(d_t^2 \mathbf{U}^{n+1}, \pmb{\Phi} \bigr)_h + (\nabla \mathbf{W}^{n+1/2}, \nabla
\pmb{\Phi}) & = \bigl( \lambda_w^{n+1} \mathbf{U}^{n+1/2},  \pmb{\Phi}\bigr)_h \,, 
\\ \nonumber
(\mathbf{W}^{n+1/2}, \pmb{\Psi})_h & = (\nabla \mathbf{U}^{n+1/2}, \nabla\pmb{\Psi}) \,,
\end{align}
where the discrete Lagrange multiplier $\lambda_w^{n+1} \in V_h$ is defined for $\bz \in \N$ as
\begin{eqnarray}
\label{lambdaw}
&& \lambda_w^{n+1}(\bz) =
\left\{ \begin{array}{ll}
0 & \mbox{if } \mathbf{U}^{n+1/2}(\bz) = {\bf 0}\,,
\\
\frac{\frac{1}{2}{\langle \mathbf{V}^0(\bz), \mathbf{V}^1(\bz)\rangle}}{\vert \mathbf{U}^{1/2}(\bz)\vert^2}
+ \frac{\bigl( \nabla \mathbf{W}^{1/2}, \mathbf{U}^{1/2}(\bz) \otimes \nabla \varphi_{\bz}\bigr)}{\beta_{\bz}  \vert \mathbf{U}^{1/2}(\bz)\vert^2}
& \text{if } n=0,\, \mathbf{U}^{1/2}(\bz) \neq {\bf 0},
\\
-\frac{\bigl\langle d_t \mathbf{U}^n(\bz), {d_t}\mathbf{U}^{n+1/2}(\bz)\bigr\rangle}{\vert \mathbf{U}^{n+1/2}(\bz)\vert^2}
+ \frac{\bigl( \nabla \mathbf{W}^{n+1/2}, \mathbf{U}^{n+1/2}(\bz) \otimes \nabla \varphi_{\bz}
\bigr)}{\beta_{\bz}  \vert \mathbf{U}^{n+1/2}(\bz)\vert^2}
& \mbox{else}\, .
\end{array}  \right.
\end{eqnarray}
The above explicit formula for the discrete Lagrange multiplier $\lambda_w^{n+1}$ guarantees that $\vert {\bf U}^{n+1}(\bz)\vert^2 = \vert {\bf U}^n(\bz)\vert^2$, see Lemma~\ref{lem_est1} below.
The formula for $n>1$ can be deduced by setting $\pmb{\Phi} = \mathbf{U}^{n+1/2}(\bz)\varphi_{\bz}$ in (\ref{scheme})
from the identity
\begin{align*}
& \frac{1}{\tau} \Bigl\langle d_t {\bf U}^{n+1}(\bz) - d_t {\bf U}^{n}(\bz), 
{\bf U}^{n+1/2}(\bz)\Bigr\rangle 
\\
&\qquad = \frac{1}{2\tau^2} \Bigl[ \vert {\bf U}^{n+1}(\bz)\vert^2 - \vert {\bf U}^n(\bz)\vert^2\Bigr]
- \frac{1}{\tau} \Bigl\langle d_t {\bf U}^n(\bz), {\bf U}^{n-1/2}(\bz) 
+ \tau d_t {\bf U}^{n+1/2}(\bz)\Bigr\rangle
\\
&\qquad  =  -\frac{1}{2\tau^2} \Bigl[ \vert {\bf U}^{n}(\bz)\vert^2 - \vert {\bf U}^{n-1}(\bz)\vert^2\Bigr]
- \Bigl\langle d_t {\bf U}^{n}(\bz), d_t {\bf U}^{n+1/2}(\bz)\Bigr\rangle 
\\
&\qquad   =  - \Bigl\langle d_t {\bf U}^{n}(\bz), d_t {\bf U}^{n+1/2}(\bz)\Bigr\rangle\, ,
\end{align*}
where we used that $\vert {\bf U}^{n+1}(\bz)\vert^2 = \vert {\bf U}^n(\bz)\vert^2 = \vert {\bf U}^{n-1}(\bz)\vert^2 = 1$.
The case $n=0$ is considered separately since $\vert {\bf U}^{-1}(\bz)\vert \neq 1$ if $|{\bf V}_0(\bz)| > 0$ and the formula can be derived analogously
using the orthogonality $\langle \mathbf{U}^{0}, \mathbf{V}^{0}\rangle = 0$.

For $n\geq 1$ we denote $\mathbf{V}^{n+1}:= d_t \mathbf{U}^{n+1}$ and for $n=0,\dots, N$ we define the discrete energy as
$$
\mathcal{E}_h \bigl({\bf V}^{n+1}, {\bf U}^{n+1}\bigr) := \frac{1}{2}\|\mathbf{V}^{n+1}\|_h^2 + \frac{1}{2} \|\Delta_h \mathbf{U}^{n+1}\|_h^2\,.
$$

\begin{lemma}\label{lem_est1}
Let $\T_h$ be a quasiuniform triangulation of $\Omega \subset \R^d$, and
$(\, {\bf U}^0, {\bf V}^0\, ) \in {\bf V}_h \times {\bf V}_h$ such that
$\vert {\bf U}^0(\bz) \vert = 1$, and
$\bigl\langle {\bf U}^0(\bz), {\bf V}^{0}(\bz)\bigr\rangle = 0$ for all $\bz \in \N$.
For $n \geq 0$, and sufficiently
small $\tilde{C} \equiv \tilde{C}(\Omega, \T_h) > 0$ independent of $k,h > 0$ such that
$\tau \leq \tilde{C} h^{2}$, there exist ${\bf U}^{n+1},\, {\bf W}^{n+1} \in {\bf V}_h$, which
solve (\ref{scheme}).
Furthermore $\vert {\bf U}^{n+1}(\bz)
\vert = 1$ for all $\bz \in \N$, and
\begin{equation}\label{discener}
 \max_{n=1,\dots,N} \mathcal{E}_h \bigl({\bf V}^{n}, {\bf U}^{n}\bigr) +
\frac{\tau^2}{2} \sum_{n=1}^N \Vert d_t {\bf V}^{n}\Vert^2_h
=
 \mathcal{E}_h \bigl( {\bf V}^0, {\bf U}^0\bigr).
\end{equation}
\end{lemma}
\begin{proof}
i) We note that the second equation in (\ref{scheme}) and (\ref{disc_lap}) imply that ${\bf W}^{n+1} = -\Delta_h {\bf U}^{n+1} \in {\bf V}_h$, $n\geq 0$;
the existence of ${\bf U}^{n+1}\in {\bf V}_h$ is shown in Lemma~\ref{lem_exist} below.

ii) The property $\vert {\bf U}^{n+1}(\bz)\vert^2 = \vert {\bf U}^{n}(\bz)\vert^2 = 1$ for all $\bz \in \N$ follows by
choosing $\pmb{\Phi} = \mathbf{U}^{n+1/2}(\bz) \varphi_z $ in (\ref{scheme})$_1$ and
noting the definition of the discrete Lagrange multiplier (\ref{lambdaw}).

iii) Energy law: we choose $\pmb{\Phi} = d_t \mathbf{U}^{n+1}$, $\pmb{\Psi} = d_t \Delta_h \mathbf{U}^{n+1}$ in (\ref{scheme}).
Noting that
$$
\bigl( \lambda_w^{n+1} \mathbf{U}^{n+1/2},  d_t\mathbf{U}^{n+1}\bigr)_h = \bigl( \lambda_w^{n+1}, |\mathbf{U}^{n+1}|^2-|\mathbf{U}^{n}|^2\bigr)_h =0,
$$
(since  $\langle \mathbf{U}^{n+1/2},  d_t\mathbf{U}^{n+1}\rangle(\mathbf{z}) = (|\mathbf{U}^{n+1}|^2-|\mathbf{U}^{n}|^2)(\mathbf{z})=0$ by part i))
we deduce that
\begin{align}\label{en1}
\bigl(d_t \mathbf{V}^{n+1}, \mathbf{V}^{n+1} \bigr)_h + (\nabla \mathbf{W}^{n+1/2}, \nabla d_t \mathbf{U}^{n+1}) & = 0 \,,
\\ \nonumber
(\mathbf{W}^{n+1/2}, d_t \Delta_h \mathbf{U}^{n+1})_h & = (\nabla \mathbf{U}^{n+1/2}, \nabla d_t \Delta_h \mathbf{U}^{n+1}) \,.
\end{align}
From the definition of $\Delta_h$ it follows that
$(\nabla \mathbf{U}^{n+1/2}, \nabla d_t \Delta_h \mathbf{U}^{n+1}) = -(\Delta_h \mathbf{U}^{n+1/2}, d_t \Delta_h \mathbf{U}^{n+1})_h$
and $(\mathbf{W}^{n+1/2}, d_t \Delta_h \mathbf{U}^{n+1})_h =  -(\nabla\mathbf{W}^{n+1/2}, \nabla d_t  \mathbf{U}^{n+1})$.
Hence, we add the respective equations in (\ref{en1}) to cancel the mixed terms and obtain after using the identitis $2(a-b)a = a^2 + (a-b)^2 - b^2$, $(a+b)(a-b)=a^2 - b^2$ that
$$
\frac{1}{2}\|\mathbf{V}^{n+1}\|_h^2 +\frac{1}{2}\|\mathbf{V}^{n+1}-\mathbf{V}^{n}\|_h^2 -\frac{1}{2}\|\mathbf{V}^{n}\|_h^2 +
\frac{1}{2}\|\Delta_h \mathbf{U}^{n+1}\|_h^2  -\frac{1}{2}\|\Delta_h \mathbf{U}^{n}\|_h^2
 = 0 \,.
$$
The discrete energy law (\ref{discener}) then follows directly after summation
of the above identity for $n=0,\dots, N-1$.

\end{proof}

The solvability of the nonlinear system given by (\ref{scheme}) is guaranteed by the following lemma.
\begin{lemma}\label{lem_exist}
Given $\mathbf{U}^{n}$, $\mathbf{U}^{n-1}\in \mathbf{V}_h$, for sufficiently small $\tau \leq C h^2$, there exist $\mathbf{U}^{n+1}$, $\mathbf{W}^{n+1}\in \mathbf{V}_h$ that satisfy (\ref{scheme}).
\end{lemma}
\begin{proof}
$i)$
From (\ref{scheme})$_2$ and (\ref{disc_lap}) we observe that $\mathbf{W}^{n} = -\Delta_h \mathbf{U}^{n}$, $\mathbf{W}^{n+1} = -\Delta_h \mathbf{U}^{n+1}$. Hence, it is enough to show the existence of
$\mathbf{U}^{n+1}$.

For {$\frac{1}{8} < \varepsilon \leq \frac{1}{4}$} we define the continuous map $\mathbf{F}_\varepsilon: \mathbf{V}_h \rightarrow \mathbf{V}_h$ via
$$
(\mathbf{F}_\varepsilon(\mathbf{u}), \pmb{\Phi}) := \frac{2}{\tau^2}\bigl( \mathbf{u} -2 \mathbf{U}^{n} + \mathbf{U}^{n-1/2}, \pmb{\Phi} \bigr)_h 
 + \big(\nabla \mathbf{w}, \nabla \pmb{\Phi}\big) - \bigl( \lambda_{w, \varepsilon}^{n+1} \mathbf{u},  \pmb{\Phi}\bigr)_h,
$$
where $\mathbf{w} := -\Delta_h \mathbf{u}$ and
\begin{align*}
\lambda_{w,\varepsilon}^{n+1}(\bz) =
-\frac{\bigl\langle d_t \mathbf{U}^n(\bz), {k^{-1}\big(\mathbf{u}(\bz) - \mathbf{U}^{n-1/2}\big)}\bigr\rangle}{ \vert \mathbf{u}(\bz)\vert_\varepsilon^2}
+ \frac{\bigl( \nabla \mathbf{w}, \mathbf{u}(\bz) \otimes \nabla \varphi_{\bz} \bigr)}{\beta_{\bz}  \vert \mathbf{u}(\bz)\vert_\varepsilon^2}
:= \lambda_1^{n+1}(\mathbf{z}) + \lambda_2^{n+1}(\mathbf{z}) \qquad \bz \in \mathcal{N}_h.
\end{align*}
with $|\cdot|_\varepsilon^2 := \max\{\vert \cdot \vert^2,\varepsilon\}$.

We estimate the first term above as
\begin{align}\label{exest1}
\nonumber
\bigl( \lambda_1^{n+1} \mathbf{u},  \mathbf{u}  \bigr)_h &  = {\sum_{\bz\in \mathcal{N}_h}} \bigl( \lambda_1^{n+1} \mathbf{u},  \mathbf{u}(\bz) {\varphi_{\bz}} \bigr)_h
\leq  {\sum_{\bz\in \mathcal{N}_h}} \frac{|\mathbf{u}(\bz)|^2}{|\mathbf{u}(\bz)|_\varepsilon^2} \frac{1}{\tau}\|\mathbf{u} - \mathbf{U}^{n-1/2}\|_h\|d_t \mathbf{U}^n\|_h
\\
& \leq \frac{1}{{2} \tau^2}\|\mathbf{u} - \mathbf{U}^{n-1/2} \|_h^2 + C \|d_t \mathbf{U}^n\|_h^2
\leq \frac{1}{\tau^2}\|\mathbf{u}\|^2 + \frac{1}{\tau^2}\|\mathbf{U}^{n-1/2} \|_h^2 + C \|d_t \mathbf{U}^n\|_h^2.
\end{align}
Noting that $\| \varphi_{\bz}\|_{L^\infty} \leq Ch^{-1}$ we estimate
\begin{equation*}
\begin{split}
\bigl( \lambda_2^{n+1} \mathbf{u},  \mathbf{u}(\bz) \varphi_{\bz} \bigr)_h
& \leq \frac{|\mathbf{u}(\bz)|^2}{|\mathbf{u}(\bz)|_\varepsilon^2} \bigl( \nabla \mathbf{w}, \mathbf{u}(\bz) \otimes \nabla \varphi_{\bz} \bigr)_{\omega_{\bz}} 
 \leq C \|\nabla \mathbf{w}\|_{\omega_{\bz}} \|\mathbf{u}(\bz) \nabla \varphi_{\bz} \|_{\omega_{\bz}}
\\
& \leq \frac{C}{h} \|\nabla \mathbf{w}\|_{\omega_\bz}  \|\mathbf{u}(\bz) \|_{\omega_\bz}.
\end{split}
\end{equation*}
Consequently, we estimate the second term using the above estimate, the inverse inequality $\|\nabla \mathbf{v}\|\leq C h^{-1}\|\mathbf{v}\|$ for $\mathbf{v} \in \mathbf{V}_h$ and (\ref{normeq}) as
\begin{equation}\label{exest2}
\begin{split}
\bigl( \lambda_2^{n+1} \mathbf{u},  \mathbf{u} \bigr)_h & = {\sum_{\bz\in \mathcal{N}_h}}  \bigl( \lambda_2^{n+1} \mathbf{u},  \mathbf{u}(\bz) \varphi_{\bz} \bigr)_h
 \leq \frac{C}{h} \sum_{\bz\in \mathcal{N}_h} \|\nabla \mathbf{w}\|_{\omega_\bz} \| \mathbf{u}(\bz)  \|_{\omega_\bz}
\\
& \leq
\frac{C}{h} \|\nabla \mathbf{w}\| \|\mathbf{u} \|_h \leq \frac{C}{h^2} \|\mathbf{w}\| \|\mathbf{u} \|_h
\leq \frac{1}{2} \|\Delta_h \mathbf{u}\|_{h}^2 + \frac{C}{h^4} \|\mathbf{u}\|_{h}^2,
\end{split}
\end{equation}
where we also used the fact that $\sum_{\bz\in \mathcal{N}_h} \|\mathbf{v}\|_{\omega_\bz} ^2 \leq  C(\mathcal{T}_h) \|\mathbf{v}\|^2$
(where the constant $C(\mathcal{T}_h)> 0$ is $h$-independent due to the regularity of the partition $\mathcal{T}_h$).

We use (\ref{exest1}), (\ref{exest2}) along with (\ref{disc_lap}) to conclude
\begin{align*}
 (\mathbf{F}_\varepsilon(\mathbf{u}), \mathbf{u})
 & \geq \frac{2}{\tau^2}\left( \|\mathbf{u} \|_h^{2} - 2|\bigl(\mathbf{U}^{n}, \mathbf{u}\bigr)_h| - |\bigl(\mathbf{U}^{n-1/2}, \mathbf{u} \bigr)_h|\right) + \|\Delta_h \mathbf{u}\|_h^2
\\
&  \quad - \frac{1}{\tau^2}\|\mathbf{u}\|_h^2  - \frac{1}{\tau^2} \|\mathbf{U}^{n-1/2}\|_h^2 - C \|d_t \mathbf{U}^n\|_h^2
- \frac{1}{2} \|\Delta_h \mathbf{u}\|_{h}^2 - \frac{C}{h^4} \|\mathbf{u} \|_{h}^2
\\
& \geq \frac{1}{\tau^2}\left( \left(1- \frac{C\tau^2}{ h^4}\right) \|\mathbf{u} \|_h^{2} - {4} \|\mathbf{U}^{n}\|_h\|\mathbf{u} \|_h - 2\|\mathbf{U}^{n-1/2}\|_h\|\mathbf{u} \|_h\right)
\\
& \quad + \frac{1}{2}\|\Delta_h \mathbf{u}\|_h^2 - \frac{1}{2\tau^2}\|\mathbf{U}^{n-1/2}\|_h^2 - C \|d_t \mathbf{U}^n\|_h^2
\\
& \geq \frac{1}{\tau^2}\|\mathbf{u}\|_h\left( \left(1- \frac{C\tau^2}{ h^4}\right) \|\mathbf{u} \|_h - 4\|\mathbf{U}^{n}\|_h - 2\|\mathbf{U}^{n-1/2}\|_h\right)
\\
& \quad + \frac{1}{2}\|\Delta_h \mathbf{u}\|_h^2 - \frac{1}{2\tau^2}\|\mathbf{U}^{n-1/2}\|_h^2 - C \|d_t \mathbf{U}^n\|_h^2.
\end{align*}
For $\tau \leq \tilde{C} h^2$ for sufficiently small $\tilde{C} \equiv \tilde{C}(\Omega)>0$ it is possible to find $R>0$ such that for all $\mathbf{u} \in\mathbf{V}_h$ with $\|\mathbf{u}\|_h \geq R$
it holds
$$
(\mathbf{F}_\varepsilon(\mathbf{u}), \mathbf{u}) \geq \frac{1}{\tau^2}\|\mathbf{u}\|_h - \frac{1}{2\tau^2}\|\mathbf{U}^{n-1/2}\|_h^2 - C \|d_t \mathbf{U}^n\|_h^2 > 0.
$$
Then the Brouwer's fixed point theorem implies the existence of $\mathbf{u} \equiv \mathbf{U}^{n+1/2}\in \mathbf{V}_h$ such that $\mathbf{F}_\varepsilon(\mathbf{U}^{n+1/2}) = 0$.

$ii)$ We show that for sufficiently small $\tau \leq C h^2$ the solution $\mathbf{U}^{n+1/2}$ from step $i)$ also satisfies $\mathbf{F}_\varepsilon(\mathbf{U}^{n+1/2}) = 0$ for $\varepsilon=0$
which implies that $\mathbf{U}^{n+1} = 2\mathbf{U}^{n+1/2} - \mathbf{U}^n$  is a solution of (\ref{scheme}).
It suffices to show that for $\tau \leq C h^2$ it holds
\begin{equation}\label{dt_small}
|\mathbf{U}^{n+1/2}(\bz)| \geq 1- \frac{\tau}{2}|\mathrm{d}_t\mathbf{U}^{n+1}(\bz)| > \frac{1}{2}\qquad \forall \bz \in \mathcal{N}_h.
\end{equation}
We proceed by induction.
Suppose that for $0\leq  j \leq n$ its holds that $\mathbf{U}^j \in \mathbf{V}_h$ satisfies $|\mathbf{U}^j(\bz)| = 1$ for all $\bz \in \mathcal{N}_h$
and that
$$
\mathcal{E}_h \bigl({\bf V}^{n}, {\bf U}^{n}\bigr) + 
\frac{\tau^2}{2} \sum_{j=1}^n \Vert d_t {\bf V}^{j}\Vert^2_h
=
 \mathcal{E}_h \bigl( {\bf V}^0, {\bf U}^0\bigr).
$$
The solution $\mathbf{U}^{n+1}$ satisfies
$$
\bigl(d_t^2 \mathbf{U}^{n+1}, \pmb{\Phi} \bigr)_h 
 + \big(\Delta_h \mathbf{U}^{n+1/2}, \Delta_h \pmb{\Phi}\big)_h = \bigl( \lambda_{w, \varepsilon}^{n+1} \mathbf{U}^{n+1/2},  \pmb{\Phi}\bigr)_h
= \bigl( (\lambda_{1}^{n+1} + \lambda_{2}^{n+1})\mathbf{U}^{n+1/2},  \pmb{\Phi}\bigr)_h.
$$
We set $\pmb{\Phi} = d_t \mathbf{U}^{n+1}$ in the above and get
$$
\frac{1}{2}d_t\left[ \|d_t \mathbf{U}^{n+1}\|_h^2 +   \|\Delta_h \mathbf{U}^{n+1}\|_h^2 \right] + \frac{\tau}{2} \|d_t^2 \mathbf{U}^{n+1}\|_h^2
= \bigl( (\lambda_{1}^{n+1} + \lambda_{2}^{n+1})\mathbf{U}^{n+1/2}, d_t \mathbf{U}^{n+1} \bigr)_h.
$$
Noting that $\tau d_t \mathbf{U}^{n+1} = 2\mathbf{U}^{n+1/2}- 2 \mathbf{U}^n$ we estimate
\begin{align*}
&(\lambda_{1}^{n+1} \mathbf{U}^{n+1/2},  d_t \mathbf{U}^{n+1})_h 
\leq 
\frac{1}{2}\left(\frac{|d_t \mathbf{U}^n|  (|d_t\mathbf{U}^{n+1}| +  |d_t\mathbf{U}^{n}|)}{\max\{|\mathbf{U}^{n+1/2}|, \varepsilon\}}, |d_t \mathbf{U}^{n+1}|\right)_h
\\
&
\leq 
\frac{1}{\tau}\left(\frac{|d_t \mathbf{U}^n| (|d_t\mathbf{U}^{n+1}| +  |d_t\mathbf{U}^{n}|)}{\max\{|\mathbf{U}^{n+1/2}|, \varepsilon\}}, |\mathbf{U}^{n+1/2}| + |\mathbf{U}^n|\right)_h
\\
& \leq
\frac{C}{\tau}\left(1 + \frac{\|\mathbf{U}^n\|_{\mathbf{L}^\infty}}{\varepsilon} \right)\|d_t \mathbf{U}^n\|_h (\|d_t\mathbf{U}^{n+1}\|_h +  \|d_t\mathbf{U}^{n}\|_h).
\end{align*}
Noting (\ref{exest2}) we estimate
\begin{align*}
&(\lambda_{2}^{n+1} \mathbf{U}^{n+1/2},  d_t \mathbf{U}^{n+1})_h 
= 
\sum_{\bz\in\mathcal{N}_h} \beta_\bz \frac{\bigl( \nabla \mathbf{W}^{n+1/2}, \mathbf{U}^{n+1/2}(\bz) \otimes \nabla \varphi_{\bz} \bigr)}{\beta_{\bz}  \max\{\vert \mathbf{U}^{n+1/2}(\bz)\vert^2, \varepsilon\}}
\mathbf{U}^{n+1/2}(\bz) d_t \mathbf{U}^{n+1}(\bz)
\\
& \leq 
C \sum_{\bz\in\mathcal{N}_h}  \frac{\big(|\nabla \mathbf{W}^{n+1/2}|,  |\nabla \varphi_{\bz}| \bigr)_{\omega_{\bz}}}{\max\{\vert \mathbf{U}^{n+1/2}(\bz)\vert^2, \varepsilon\}}
|\mathbf{U}^{n+1/2}(\bz)|^2  |d_t \mathbf{U}^{n+1}(\bz)|
\leq \frac{C}{h^2} \|\Delta_h \mathbf{U}^{n+1/2}\|_h\| d_t \mathbf{U}^{n+1}\|_h,
\end{align*}
where we used
\begin{align*}
\big(|\nabla \mathbf{W}^{n+1/2}|,  |d_t \mathbf{U}^{n+1}(\bz)||\nabla \varphi_{\bz}| \bigr)_{\omega_{\bz}}
 & \leq \|\nabla \mathbf{W}^{n+1/2}\|_{L^2(\omega_\bz)}\||d_t \mathbf{U}^{n+1}(\bz)| |\nabla \varphi_{\bz}|\|_{L^2(\omega_\bz)}
\\
& \leq \frac{1}{h^2}\|\mathbf{W}^{n+1/2}\|_{L^2(\omega_\bz)}\||d_t \mathbf{U}^{n+1}(\bz)| |\varphi_{\bz}|\|_{L^2(\omega_\bz)}.
\end{align*}
Hence, we conclude that
\begin{align*}
& \frac{1}{2\tau} \left(1 -\frac{1}{4} - \frac{C\tau}{h^2}\right)\|d_t \mathbf{U}^{n+1}\|_h^2
   + \frac{1}{2\tau}\left(1-\frac{C\tau}{h^2}\right)\|\Delta_h \mathbf{U}^{n+1}\|_h^2  + \frac{\tau}{2} \|d_t^2 \mathbf{U}^{n+1}\|_h^2
\\
& \leq \frac{1}{2\tau}\left(1 + C\left(1 + \frac{\|\mathbf{U}^n\|_{\mathbf{L}^\infty}}{\varepsilon} \right)^2\right) \|d_t \mathbf{U}^{n}\|_h^2
 + \frac{1}{2\tau}\left(1+\frac{C\tau}{h^2}\right)\|\Delta_h \mathbf{U}^{n}\|_h^2.
\end{align*}
By the inverse estimate $\|v_h\|_{L^\infty}\leq h^{-d/2} \|v_h\|$, recalling the induction assumption, we deduce for sufficiently small $\tau \leq C h^2$ that
\begin{align*}
 \frac{\tau^2}{2} \|d_t \mathbf{U}^{n+1}\|_{L^\infty}^2
& \leq C \tau^2 h^{-d} \|d_t \mathbf{U}^{n+1}\|_h^2
\\
&
\leq \frac{C\tau^2 h^{-d}}{4}\left[\left(1 + C\left(1 + \frac{\|\mathbf{U}^n\|_{\mathbf{L}^\infty}}{\varepsilon} \right)^2\right) \|d_t \mathbf{U}^{n}\|_h^2 + 2\|\Delta_h \mathbf{U}^{n}\|_h^2\right]
\leq \tilde{C}\tau^2 h^{-d} < \frac{1}{2}.
\end{align*}
It follows that (\ref{dt_small}) is satisfied for sufficiently small $\tau \leq C h^{d/2} \leq C h^{2}$ for $d\leq 3$.
\end{proof}
\begin{remark}
Compared to the (second order) wave map equation, which requires a time step restriction $\tau \leq C h^{\max\{1, d/2\}}$ for the existence of
the discrete solution \cite{num_harm09,num_move},
the bi-harmonic wave map equation requires a stronger restriction $\tau \leq C h^{2}$ (see Lemma~\ref{lem_exist}).
\rev{We remark that the fixed-point algorithm introduced in Section~\ref{sec_fixpnt} converges under the weaker condition $\tau \leq C h$
in practical computations, but its theoretical convergence analysis is still open
(similarly to the wave map case \cite{num_move})}.
\end{remark}

\section{Convergence of the numerical approximation}\label{sec_conv}

Given the solution $\{\mathbf{U}^n\}_{n=0}^{N}\subset \mathbf{V}_h$
we consider the interpolants ${\bf u}_{\tau,h}, {\bf u}^-_{\tau,h}, \overline{{\bf u}}_{\tau,h}: \Omega_T \rightarrow \R^3$ such that for $t \in [t_n, t_{n+1}) \times \Omega$
\begin{eqnarray}\nonumber
&&{\bf u}_{\tau,h}(t) = \frac{t - t_n}{\tau} {\bf U}^{n+1} +
\frac{t_{n+1}-t}{\tau} {\bf U}^n \,,  \\ \label{notation}
&&{\bf u}^-_{\tau,h}(t) = {\bf U}^{n}\,, \qquad
{\bf u}^+_{\tau,h}(t) = {\bf U}^{n+1}\,, \\ \nonumber
&&\overline{{\bf u}}_{\tau,h}(t) := {\bf U}^{n+1/2},
\end{eqnarray}
and analogously we define the corresponding interpolants for $\{\mathbf{V}^{n}\}_{n=0}^N$ ($\mathbf{V}^{n} = \mathrm{d}_t \mathbf{U}^{n}$ for $n\geq 1$), and $\{\mathbf{W}^{n}\}_{n=0}^N$
($\mathbf{W}^{n} = -\Delta_h \mathbf{U}^{n}$ for $n\geq 1$).
In the sequel we replace the subscript $\tau, h$ by $h$ to simplify the notation.
The scheme (\ref{scheme}) can be equivalently written in terms of the above interpolants for $\pmb{\Phi} \in {C^\infty_0}([0,T]; \mathbf{V}_h)$ as
\begin{equation}\label{schemeh}
\begin{split}
\int_0^T \left[\big(\partial_t {\bf v}_{h}, \pmb{\Phi}\big)_h + \big(\nabla \overline{{\bf w}}_{h}, \nabla \pmb{\Phi}\big) - \big(\lambda_h^{+}\overline{{\bf u}}_h, \pmb{\Phi}\big)_h\right]\mathrm{d}t = 0,
\\
\int_0^T \left[\big(\overline{{\bf w}}_{h}, \pmb{\Phi}\big) - \big(\nabla \overline{{\bf u}}_h, \nabla \pmb{\Phi}\big)_h\right]\mathrm{d}t = 0.
\end{split}
\end{equation}

In the next lemma we restate (\ref{schemeh}) in a form that is suitable for showing convergence to the weak formulation (\ref{weakform}).
\begin{lemma}\label{lem_reform}
For all $\pmb{\Psi}\in  C^\infty_0([0,T); C^\infty(\Omega;\mathbb{R}^3))$ it holds that
\begin{align}\label{schemeh2}
& \left|\int_0^T \left[-\big(\partial_t {\bf u}_{h}, \partial_t [{\bf u}_h \times \pmb{\Psi}]\big)_h + \big(\Delta_h \overline{{\bf u}}_{h}, \Delta_h \mathcal{I}_h [\overline{{\bf u}}_h \times \pmb{\Psi}]\big)_h \right]\mathrm{d}t
- \big(\mathbf{V}^0, \mathbf{U}^0\times \pmb{\Psi}(0)\big)_h \right|
\\ \nonumber
&\qquad  \leq \left| \int_0^T \big([{\bf v}_{h} - {\bf v}_{h}^+], \partial_t [{\bf u}_h \times \pmb{\Psi}]\big)_h \mathrm{d} t\right| 
+ C\tau^{1/2}\mathcal{E}_h(\mathbf{V}^0, \mathbf{U}^0)^{1/2}\|\pmb{\Psi}\|_{L^\infty(\Omega_T)}.
\end{align}
\end{lemma}
\begin{proof}
We take $\pmb{\Phi} = \mathcal{I}_h [\overline{{\bf u}}_h \times \pmb{\Psi}]$ with $\pmb{\Psi} \in C^\infty_0([0,T); C^\infty(\Omega;\mathbb{R}^3)$ in (\ref{schemeh})
and note that the term containing $\lambda_h^{+}$ disappears by orthogonality.
Next, we rewrite the first term as
\begin{equation}\label{est_vt}
\big(\partial_t {\bf v}_{h}, \overline{{\bf u}}_h \times \pmb{\Psi}\big)_h = \big(\partial_t {\bf v}_{h}, {\bf u}_h \times \pmb{\Psi} \big)_h + \big(\partial_t {\bf v}_{h}, (\overline{{\bf u}}_h - {\bf u}_h) \times \pmb{\Psi} \big)_h.
\end{equation}
Noting that
$$
\overline{{\bf u}}_h - {\bf u}_h =\Big(\frac{1}{2}(t_{n+1}+ t_n) - t\Big) {\bf v}_h^+\qquad \text{for } t\in [t_n, t_{n+1}),
$$
we estimate
$\big(\partial_t {\bf v}_{h}, (\overline{{\bf u}}_h - {\bf u}_h) \times \pmb{\Psi} \big)_h \leq \frac{\tau}{2}\|{\partial_t {\bf v}_{h}}\|_h\|{\bf v}_h^+\|_h\|\pmb{\Psi}\|_{L^\infty}$.
Hence, we use the discrete energy estimate (\ref{discener}) to bound the second term in (\ref{est_vt}) as
\begin{align*}
\int_0^T \big(\partial_t {\bf v}_{h}, (\overline{{\bf u}}_h - {\bf u}_h) \times \pmb{\Psi} \big)_h\mathrm{d}t & \leq 
C \tau^{1/2} \left(\int_0^T\tau \|{ \partial_t {\bf v}_{h}}\|_h^2\mathrm{d} t\right)^{1/2}\|\pmb{\Psi}\|_{L^\infty(\Omega_T)}
\\
 & \leq C\tau^{1/2}\mathcal{E}_h(\mathbf{V}^0, \mathbf{U}^0)^{1/2}\|\pmb{\Psi}\|_{L^\infty(\Omega_T)}.
\end{align*}
We use integration by parts and that $\partial_t {\bf u}_h = {\bf v}_h^+$ to rewrite the the first term in (\ref{est_vt}) as
\begin{align*}
 \int_0^T \big(\partial_t {\bf v}_{h}, {\bf u}_h \times \pmb{\Psi} \big)_h \mathrm{d} t  = 
- \int_0^T \big({\bf v}_{h}, \partial_t [{\bf u}_h \times \pmb{\Psi}]\big)_h \mathrm{d} t - \big(\mathbf{V}^0, \mathbf{U}^0\times \pmb{\Psi}(0)\big)_h
\\
\qquad = 
 \int_0^T \big(- \partial_t {\bf u}_{h} - [{\bf v}_{h} - {\bf v}_{h}^+], \partial_t [{\bf u}_h \times \pmb{\Psi}]\big)_h \mathrm{d} t - \big(\mathbf{V}^0, \mathbf{U}^0\times \pmb{\Psi}(0)\big)_h.
\end{align*}
By the definition of the discrete Laplacian (\ref{disc_lap}) we note that the second equation in (\ref{schemeh}) implies that $\overline{{\bf w}}_{h} = -\Delta_h \overline{{\bf u}}_{h}$.
Hence, using again (\ref{disc_lap}) we rewrite
$$
\big(\nabla \overline{{\bf w}}_{h}, \nabla \mathcal{I}_h [\overline{{\bf u}}_h \times \pmb{\Psi}] \big)
= \big(\overline{{\bf w}}_{h}, -\Delta_h \mathcal{I}_h [\overline{{\bf u}}_h \times \pmb{\Psi}] \big)_h = \big(\Delta_h \overline{{\bf u}}_{h}, \Delta_h \mathcal{I}_h [\overline{{\bf u}}_h \times \pmb{\Psi}]\big)_h.
$$
The inequality (\ref{schemeh2}) then follows from the last three identities above.
\end{proof}

Next, we state a discrete analogue of the Aubin-Lions-Simin compactness lemma for the embedding $H^2\subset H^1$,
which provides strong $H^1$ convergence of the numerical approximation. This strong convergence is required to identify the limits of the nonlinear terms in the numerical scheme.
\begin{lemma}\label{lem_compact}
Let $\{u_h\}_{h>0}$ be a sequence in $\{L^2(0,T,V_h)\}_{h>0}$ such that
\begin{align*}
\int_{0}^T\|u_h\|_{h}^2\mathrm{d}t \leq C, \quad \int_{0}^T\|\Delta_h u_h\|_h^2\mathrm{d}t  \leq C \text{ and }
\int_{0}^T\|\partial_t u_h\|_h^2\mathrm{d}t  \leq C.
\end{align*}
Then also $\D \int_{0}^T\|\nabla u_h\|_h^2\mathrm{d}t  \leq C$ for $h>0$ and there exists a limit $u\in L^2(0,T; H^2(\Omega))$ and a subsequence such that
\begin{align*}
u_h & \rightarrow u   &&   \text{ in } L^2(0,T; H^{1}(\Omega)),\\
\Delta_h u_h & \rightharpoonup \Delta u && \text{in } L^2(0,T; L^2(\Omega)).
\end{align*}
\end{lemma}
\begin{proof}
The equivalence of norms (\ref{normeq})
implies that $\{u_h\}$, $\{\Delta_h u_h\}$ are bounded in $L^2(0,T; L^2(\Omega))$.
Moreover, noting (\ref{convdiscnorm}) we bound by the Cauchy-Schwarz and Young inequalities
\begin{align}\label{h1_dbnd}
\|\nabla u_h\|^2 & = (\nabla u_h, \nabla u_h) = -(\Delta_h u_h, u_h)_h \leq |(\Delta_h u_h, u_h)| + Ch \|\Delta_h u_h\| \|\nabla u_h\|
\\ \nonumber
& \leq (1+C)\|\Delta_h u_h\| \| u_h\| \leq C\left(\|\Delta_h u_h\|^2 +  \| u_h\|^2\right),
\end{align}
where we also used the inverse inequality $\|\nabla u_h\| \leq C h^{-1}\|u_h\|$ and (\ref{convdiscnorm}).

The above inequality along with the bounds for $\{u_h\}$, $\{\Delta_h u_h\}$ implies the sequence $\{u_h\}$ is also bounded in $L^2(0,T; H^{1}(\Omega))$.
Hence, we deduce that there exist a subsequence (not relabeled) such that $u_h\rightharpoonup u$ in $L^2(0,T; H^{1}(\Omega))$.
Moreover, since $ \partial_t u_h$ is bounded in $L^2(0,T; L^2(\Omega))$ a standard Aubin-Lions-Simon result, see, e.g., \cite[Theorem 2.1, Ch. III]{book_temam},
implies strong convergence $u_h \rightarrow u$ in $L^2(0,T; L^2(\Omega))$.

The bound $\|\Delta_h u_h\|_{L^2(0,T;L^2)} \leq C$
implies the existence of a limit $w\in L^2(0,T; L^2(\Omega))$ such that $\Delta_h u_h \rightharpoonup w$ in $L^2(0,T; L^2(\Omega))$.
By the weak convergence of  $\{u_h\}$ in $L^2(0,T; H^{1}(\Omega))$ we deduce that
$w=\Delta u$ since (taking a strongly converging sequence $\varphi_h \rightarrow \varphi$, e.g., $\varphi_h = \mathcal{I}_h \varphi$)
\begin{align*}
& \int_0^T (\nabla u, \nabla \phi) \mathrm{d}t  = \lim_{h\rightarrow 0}\int_0^T (\nabla u_h, \nabla \phi_h) \mathrm{d}t = \lim_{h\rightarrow 0} \int_0^T  (-\Delta_h u_h, \phi_h)_h\mathrm{d}t
\\
& = \lim_{h\rightarrow 0} \int_0^T  \big[(-\Delta_h u_h, \phi_h) + \mathcal{O}(h)\big]\mathrm{d}t = \lim_{h\rightarrow 0} \int_0^T  (-\Delta_h u_h, \phi)\mathrm{d}t = \int_0^T (-w,\phi)\mathrm{d}t,
\end{align*}
where we used that (\ref{convdiscnorm}) implies
$$
 |(-\Delta_h u_h, \phi_h)_h - (-\Delta_h u_h, \phi_h)| \leq Ch \|\Delta_h u_h\| \|\nabla \phi_h\|.
$$
Since $\nabla u_h \rightharpoonup \nabla u$ the strong convergence of the gradient follows if we show the convergence of the norm
$\lim_{h\rightarrow 0}\int_0^T \|\nabla u_h\|^2 \mathrm{d}t =  \int_0^T \|\nabla u\|^2  \mathrm{d}t$.
Consequently, using that $\Delta_h u_h\rightharpoonup \Delta u$, $u_h \rightarrow u$ in $L^2(0,T; L^2(\Omega))$ and (\ref{convdiscnorm}) we conclude that
\begin{align*}
& \lim_{h\rightarrow 0} \int_0^T \|\nabla u_h\|^2 \mathrm{d}t = \lim_{h\rightarrow 0} \int_0^T (\nabla u_h, \nabla u_h)\mathrm{d}t
= \lim_{h\rightarrow 0} \int_0^T (-\Delta u_h, u_h)_h\mathrm{d}t
\\
& = \int_0^T (-\Delta u, u)\mathrm{d}t = \int_0^T \|\nabla u\|^2\mathrm{d}t.
\end{align*}
\end{proof}

Next we formulate a (perturbed) discrete product rule for the discrete Laplace operator. The lemma enables us to use discrete orthogonality property in the sequel,
which is required to identify the limit of the second term in on the left-had side of (\ref{schemeh2}).
\begin{lemma}\label{lem_product}
The following discrete product rule holds for the discrete Laplacian (\ref{disc_lap}) on right-angled triangulations:
\begin{align*}
& \left|(-\Delta_h \mathcal{I}_h [v_h w], u_h)_h - \big[(-\Delta_h v_h)\mathcal{I}_h w, u_h)_h + (v_h (-\Delta_h \mathcal{I}_h w), u_h)_h
-2(\nabla v_h \nabla\mathcal{I}_h w, u_h)\big]\right|
\\
&\qquad \leq C  h \|\nabla u_h\|\|\nabla v_h\|\|\nabla w\|_{L^\infty(\Omega)}  .
\end{align*}
for any $ u_h,v_h \in V_h$ and $w\in C(\overline{\Omega})\cap W^{1,\infty}(\Omega)$.
\end{lemma}
\begin{proof}
$i)$ We first consider Type-$1$ triangulations.
We employ the transformation on the reference simplex $\hat{T} \subset \mathbb{R}^d$ with vertices $\hat{p}_0 = {\bf 0}$, $\hat{p}_i = \alpha_i \mathbf{e}_i$, $i=1,\dots,d$.
Since $\mathcal{T}_h$ is right-angled, to each element $T\in \mathcal{T}_h$, there is a rotation matrix $R_T$ such that the map
$x = \mathcal{R}_T(\hat{x}) = p_0 + R_T \hat{x}$ maps $\hat{T}$ onto $T$ (for suitable $\alpha_i$).
We denote $\hat{u}(\hat{x})= u(\mathcal{R}_T(\hat{x})) =u(x)$ and note the identity $\nabla_x u = R_T \nabla_{\hat{x}} \hat{u}$, since $R_T^T = R_T^{-1}$.
We note that for $\hat{u} \in \mathcal{P}^1(\hat{T})$ it holds that
$$
\partial_{\hat{x}_i} \hat{u}|_{\hat{T}} = \frac{\hat{u}(\hat{p}_i)-\hat{u}(\hat{p}_0)}{\alpha_i}.
$$
Using the definition of the discrete Laplacian we also note that
\begin{equation}\label{deltahat}
(-\Delta_h \mathcal{I}_h [v_h w], \varphi_\bz)_h = (\nabla \mathcal{I}_h [v_h w], \nabla \varphi_\bz) =  \sum_{T\subset \omega_z} (\nabla \mathcal{I}_h [v_h w], \nabla \varphi_\bz)_{T}.
\end{equation}
Using that $\nabla_x u = R_T \nabla_{\hat{x}} \hat{u} =: R_T \hat{\nabla} \hat{u}$ and that $R_T$ is a rotation we observe that
$$
(\nabla \mathcal{I}_h [v_h w], \nabla \varphi_\bz)_{T} = (R_T \hat{\nabla} \mathcal{I}_h [\hat{v}\hat{w}], R_T \hat{\nabla} \hat{\varphi}_{\hat{\bz}})_{\hat{T}}
= (\hat{\nabla} \mathcal{I}_h [\hat{v}\hat{w}], \hat{\nabla} \hat{\varphi}_{\hat{\bz}})_{\hat{T}}.
$$
Hence, we may proceed by considering the the reference element $\hat{T}$, below we drop the $\hat{}$ to simplify the notation.

On each interior triangle there are $d+1$ non-zero basis functions associated with its nodes $p_0,\dots, p_d$ which we denote as  $\varphi_\bz$, $\bz = p_0,\dots, p_d$.
We note that if $\bz = p_0$
$$
\nabla \varphi_\bz = -(\alpha_1^{-1},\dots, \alpha_d^{-1})^T,
$$
and for $\bz = p_i$ with $i> 0$ it holds
$$
\nabla \varphi_\bz = (0, \dots, 0, \alpha_i^{-1},0\dots,0)^T.
$$
Below we consider the two above cases separately.

For $\bz = p_i$, $i> 0$ we rewrite
\begin{align*}
& \{\nabla \mathcal{I}_h [v_h w] \}_i = \partial_{x_i} [v_h w] = \frac{[v_h w](p_i)-[v_h w](p_0)}{\alpha_i}
= \frac{v_h(p_i) (w(p_i)-w(p_0))}{\alpha_i} + \frac{(v_h(p_i)-v_h(p_0)) w (p_0)}{\alpha_i}
\\
& = \frac{v_h(p_i) (w(p_i)-w(p_0))}{\alpha_i} + \frac{(v_h(p_i)-v_h(p_0)) w (p_i)}{\alpha_i} -  \frac{(v_h(p_i)-v_h(p_0)) (w(p_i)- w(p_0))}{\alpha_i}
\end{align*}

Hence, we deduce that
\begin{align}\label{prod1}
\langle \nabla \mathcal{I}_h [v_h w], \nabla \varphi_\bz\rangle = v_h(\bz)\langle \nabla\mathcal{I}_h w, \nabla \varphi_\bz\rangle  + w(\bz)\langle \nabla v_h, \nabla \varphi_\bz\rangle
- \partial_{x_i} v_h \partial_{x_i} \mathcal{I}_h w. 
\end{align}

In the case that $\bz = p_0$ we proceed similarly as in the previous case and arrive at
\begin{align}\label{prod2}
\langle \nabla \mathcal{I}_h [v_h w], \nabla \varphi_\bz\rangle = v_h(\bz)\langle \nabla\mathcal{I}_h w, \nabla \varphi_\bz\rangle  + w(\bz)\langle \nabla v_h, \nabla \varphi_\bz\rangle
- \langle  \nabla v_h, \nabla \mathcal{I}_h w\rangle. 
\end{align}
For each $T\in \mathcal{T}_h$, $\overline{T}=\mathrm{conv}\{p_0,\dots, p_d\}$, noting that $u_h|_T = \sum_{i=0}^d u_{p_i} \varphi_{p_i}:= \sum_{i=0}^d u_i \varphi_i$, we deduce from (\ref{prod1}), (\ref{prod2}) that
\begin{align*}
& 
(\nabla \mathcal{I}_h [v_h w], \nabla u_h)_{T}
= \sum_{i = 0}^{d} u_{i}(\nabla \mathcal{I}_h [v_h w], \nabla \varphi_i)_{T}
\\
& = \sum_{i=0}^{d} \Big( v_i u_i\langle \nabla\mathcal{I}_h w, \nabla \varphi_i\rangle|T| + w_iu_i\langle \nabla v_h, \nabla \varphi_i\rangle|T|\Big)
- u_0 \langle  \nabla v_h, \nabla \mathcal{I}_h w\rangle|T| -  \sum_{i=1}^{d}u_i \partial_{x_i} v_h \partial_{x_i} \mathcal{I}_h w |T|
\\
& := I_T - II_T - III_T.
\end{align*}
By the definition of the discrete Laplacian we deduce that
\begin{align}\label{est_it}
\nonumber
\sum_{T\in\mathcal{T}_h}I_T & = \sum_{T\in\mathcal{T}_h} \left(\big(\nabla\mathcal{I}_h w, \nabla [u_h v_h]\big)_{T} + \big(\nabla v_h, \nabla \mathcal{I}_h [u_h w]\big)_{T}\right)
=  \big(\nabla\mathcal{I}_h w, \nabla [u_h v_h]\big) + \big(\nabla v_h, \nabla \mathcal{I}_h [u_h w]\big)
\\
& = (v_h (-\Delta_h \mathcal{I}_h w), u_h)_h + ((-\Delta_h v_h)\mathcal{I}_h w, u_h)_h.
\end{align}

Next, we consider the term $III_T$
$$
III_T = \sum_{i=1}^{d}u_i \partial_{x_i} v_h \partial_{x_i} \mathcal{I}_h w |T|
= u_0  \langle  \nabla v_h, \nabla \mathcal{I}_h w\rangle |T| + \sum_{i=1}^{d} (u_i - u_0) \partial_{x_i} v_h \partial_{x_i} \mathcal{I}_h w |T|
:= III_{T,1} + III_{T,2}.
$$
We estimate the second term as
$$
III_{T,2} = \sum_{i=1}^{d} \alpha_i \partial_{x_i} u_h \partial_{x_i} v_h \partial_{x_i} \mathcal{I}_h w |T|
\leq C h \|\nabla u_h\|_T \|\nabla v_h \|_T \|\nabla w\|_{L^\infty(T)}.
$$
We add $III_{T,1}$ to $II_T$ and get
\begin{align*}
& II_T+III_{T,1} = 2 u_0  \langle  \nabla v_h, \nabla \mathcal{I}_h w\rangle|T| = 2 \sum_{i=0}^d\frac{u_0}{d+1} \langle  \nabla v_h, \nabla \mathcal{I}_h w\rangle|T|
\\
&
= 2 \sum_{i=0}^d\frac{u_i}{d+1}\langle  \nabla v_h, \nabla \mathcal{I}_h w\rangle|T| +
2\sum_{i=1}^d \frac{(u_0 - u_i)}{d+1}\langle  \nabla v_h, \nabla \mathcal{I}_h w\rangle|T|
\\
&
= 2 (\nabla v_h\nabla \mathcal{I}_h w, u_h)_{T} +
2\sum_{i=1}^d \frac{(u_0 - u_i)}{d+1}\langle  \nabla v_h, \nabla \mathcal{I}_h w\rangle|T|,
\end{align*}
where we used $\varphi_i \equiv \varphi_{p_i}(p_i) = 1$ to deduce the last equality.

We estimate the last term on the right-hand side in the above equality as
$$
\sum_{i=1}^d (u_0 - u_i)\langle  \nabla v_h, \nabla \mathcal{I}_h w\rangle|T|
=
\sum_{i=1}^d \alpha_i \partial_{x_i} u_h \langle  \nabla v_h, \nabla \mathcal{I}_h w\rangle|T|
\leq C h \|\nabla u_h\|_T\|\nabla v_h\|_T\|\nabla w\|_{L^\infty(T)}
$$
Collecting the above estimates for $II_T$, $III_T$, summing over $T\in\mathcal{T}_h$,
noting (\ref{est_it}) and (\ref{deltahat}) along with the identity
$(-\Delta_h \mathcal{I}_h [v_h w], u_h)_h = \sum_{\bz\in\mathcal{N}_h}(\nabla \mathcal{I}_h [v_h w], u_h(\bz)\nabla \varphi_\bz)$
concludes the proof.

$ii)$ For Type-$2$ triangulations we may proceed analogously as in the first part (recall that in this case $d=3$).
For simplicity we assume a uniform partition with mesh size $h$, a generalization to non-uniform partitions is straightforward.
As in the first part, by rotation and translation, all tetrahedra can be mapped onto a reference simplex $\hat{T} = \mathrm{conv}\{\hat{p}_0, \dots, \hat{p}_3\}$
where $\hat{p}_0 = {\bf 0}$, $\hat{p}_1 = h \mathbf{e}_1$, $\hat{p}_2 = h(\mathbf{e}_1 + \mathbf{e}_2)$ and $\hat{p}_3 = h\mathbf{e}_3$.
Consequently we observe that for $\hat{u} \in \mathcal{P}^1(\hat{T})$ it holds that
\begin{align*}
\partial_{\hat{x}_1} \hat{u}|_{\hat{T}}  = \frac{\hat{u}(\hat{p}_1)-\hat{u}(\hat{p}_0)}{h},
\quad
\partial_{\hat{x}_2} \hat{u}|_{\hat{T}}  = \frac{\hat{u}(\hat{p}_2)-\hat{u}(\hat{p}_1)}{h},
\quad
\partial_{\hat{x}_3} \hat{u}|_{\hat{T}}  = \frac{\hat{u}(\hat{p}_1)-\hat{u}(\hat{p}_0)}{h}.
\end{align*}
The rest of the proof follows as in part $i)$ with minor modifications, therefore we omit it.
\end{proof}

\begin{remark}\label{rem_product}
A closer inspection of the proof of Lemma~\ref{lem_product} reveals that in the one dimensional case $\Omega \subset \mathbb{R}^1$ the discrete product rule simplifies to an equality
which is analogous to the continuous product rule:
\begin{align}\label{prod_1d}
(-\Delta_h \mathcal{I}_h [v_h w], u_h)_h = (-\Delta_h v_h)\mathcal{I}_h w, u_h)_h + (v_h (-\Delta_h \mathcal{I}_h w), u_h)_h -2(\nabla v_h \nabla\mathcal{I}_h w, u_h).
\end{align}
\end{remark}

The discrete product rule (\ref{prod_1d}) is used in the next theorem to conclude convergence of the numerical approximation for $d=1$.
\begin{theorem}[Convergence]\label{thm_converg}
Let Assumption~1 hold and let $\Omega\subset\mathbb{R}^1$. Furthermore, let $\mathbf{U}^0 \rightarrow \mathbf{u}_0$, $\mathbf{V}^0 \rightarrow \mathbf{v}_0$ for $h\rightarrow 0$,
{$\mathcal{E}_h(\mathbf{U}^0, \mathbf{V}^0)\leq C$ and $|{\bf U}^0({\bf z})| = 1$, $\langle\mathbf{U}^0({\bf z}), \mathbf{V}^0({\bf z}) \rangle = 0$  for ${\bf z}\in\mathcal{N}_h$}.
Then there exists a weak solution ${\bf u}: [0,T]\times \Omega \rightarrow \mathbb{S}^2$ of (\ref{biwave}) according to Definition~\ref{def_weak}
and a subsequence $\{{\bf u}_h\}_{h,\tau}$ such that for $(\tau,h)\rightarrow 0$
\begin{align*}
{\bf u}_h & \overset{\ast}{\rightharpoonup} {\bf u} && \text{ in } L^{\infty}(0,T; {\bf H}^{1}),
\\
\partial_t {\bf u}_h & \overset{\ast}{\rightharpoonup} \partial_t {\bf u} && \text{  in } L^{\infty}(0,T; {\bf L}^{2}).
\end{align*}
\end{theorem}
\begin{proof}
Estimate (\ref{discener}) and Lemma~\ref{lem_compact} implies the boundedness of the sequence $\{{\bf u}_h\}_{h,\tau>0}$
in $L^{\infty}(0,T; {\bf H}^{1})$. This together with the bound on $\{\partial_t {\bf u}_h \}_{h,\tau>0}$ implies the (sub)convergence
\begin{align}
\label{weakconv}
{\bf u}_h,{\bf u}_h^+,\overline{{\bf u}}_h & \overset{\ast}{\rightharpoonup} {\bf u} &&  \text{  in } L^{\infty}(0,T; {\bf H}^{1}),
\\
\label{strongconv}
{\bf u}_h,{\bf u}_h^+,\overline{{\bf u}}_h & \rightarrow {\bf u} && \text{  in } L^{2}(0,T; {\bf L}^{2}),
\\
\partial_t {\bf u}_h,{\bf v}_h,{\bf v}_h^+ & \overset{\ast}{\rightharpoonup} \partial_t {\bf u} && \text{  in } L^{\infty}(0,T; {\bf L}^{2}),
\\
\label{weakconv_deltah}
\Delta_h {\bf u}_h, \Delta_h \overline{{\bf u}}_h, \Delta_h {\bf u}_h^+ & \overset{\ast}{\rightharpoonup} \Delta {\bf u} && \text{  in } L^{\infty}(0,T; {\bf L}^{2}),
\end{align}
where to show the convergence of ${\bf v}_h$, ${\bf v}_h^+$ to the same limit for $\tau,h\rightarrow 0$ can be concluded thanks to the the numerical dissipation term in (\ref{discener}), since,
\begin{equation}\label{num_damp}
\|{\bf v}_h-{\bf v}_h^+\|_{L^2(0,T; {\bf L}^2)}^2 \leq C \tau^3\sum_{n=1}^N\|d_t \mathbf{V}^n\|_h^2 \leq C \tau \rightarrow 0.
\end{equation}
Similarly we also conclude that the limits of $\Delta_h {\bf u}_h, \Delta_h \overline{{\bf u}}_h, \Delta_h {\bf u}_h^+$ coincide.

Furthermore, Lemma~\ref{lem_compact} implies the strong convergence
\begin{equation}\label{conv_grad}
\nabla {\bf u}_h, \nabla \overline{{\bf u}}_h, \nabla {\bf u}_h^+  \rightarrow \nabla {\bf u} \qquad \text{  in } L^{2}(0,T; {\bf L}^{2}).
\end{equation}

Since $|{\bf u}_h^+({\bf z})| = 1$ for any $T\in\mathcal{T}_h$  and ${\bf z}\in \mathcal{N}_h$ we have for ${\bf x}, {\bf z} \in T$ that $
|{\bf u}_h^+({\bf x})| - 1| = ||{\bf u}_h^+({\bf x})| - |{\bf u}_h^+({\bf z})| | \leq \mathrm{diam}(T) |\nabla {\bf u}_h^+|_T|$. Hence we obtain
$$
\big\||{\bf u}_h^+| - 1\big\|_{L^2(T)} \leq C h \|\nabla {\bf u}_h^+\|_{L^2(T)}.
$$
This implies that  $|{\bf u}_h^+| \rightarrow 1$ almost everywhere in $\Omega_T$ for $(\tau, h)\rightarrow 0$
and consequently $|{\bf u}|=1$ almost everywhere in $\Omega_T$.

To show that the limiting function ${\bf u}$ is a weak solution we pass to the limit in (\ref{schemeh2}) for $(\tau,h)\rightarrow 0$.
Noting the orthogonality of $\Delta_h  \overline{{\bf u}}_h$ and $\Delta_h  \overline{{\bf u}}_h\times \mathcal{I}_h  \pmb{\Psi} $
we deduce by the discrete product rule (\ref{prod_1d}) that 
\begin{equation}\label{lap_terms}
\big(\Delta_h \overline{{\bf u}}_{h}, \Delta_h \mathcal{I}_h [\overline{{\bf u}}_h \times \pmb{\Psi}]\big)_h
=
\big(\Delta_h \overline{{\bf u}}_{h}, \overline{{\bf u}}_h \times \Delta_h \mathcal{I}_h \pmb{\Psi}\big)_h + 2 \big(\Delta_h \overline{{\bf u}}_{h}, \nabla  \overline{{\bf u}}_h \times \nabla \mathcal{I}_h \pmb{\Psi}\big).
\end{equation}


To estimate the first term in (\ref{lap_terms}) we bound
\begin{align*}
& |\big(\Delta_h \overline{{\bf u}}_{h}, \overline{{\bf u}}_h \times \Delta_h \mathcal{I}_h \pmb{\Psi}\big)_h - \big(\Delta {\bf u}, {\bf u} \times \Delta  \pmb{\Psi}\big)|
\leq
 |\big(\Delta_h \overline{{\bf u}}_{h}, \overline{{\bf u}}_h \times \Delta_h \mathcal{I}_h \pmb{\Psi}\big)_h - \big(\Delta_h \overline{{\bf u}}_{h}, \overline{{\bf u}}_h \times \Delta \pmb{\Psi}\big)_h|
\\
& \qquad
+ |\big(\Delta_h \overline{{\bf u}}_{h}, \overline{{\bf u}}_h \times \Delta\pmb{\Psi}\big)_h - \big(\Delta_h \overline{{\bf u}}_{h}, \overline{{\bf u}}_h \times \mathcal{I}_h \Delta \pmb{\Psi}\big)|
+ |\big(\Delta_h \overline{{\bf u}}_{h}, \overline{{\bf u}}_h \times \mathcal{I}_h  \Delta\pmb{\Psi}\big) - \big(\Delta_h \overline{{\bf u}}_{h}, \overline{{\bf u}}_h \times \Delta \pmb{\Psi}\big)|
\\
&
\qquad
+ |\big(\Delta_h \overline{{\bf u}}_{h}, \overline{{\bf u}}_h \times \Delta\pmb{\Psi}\big) - \big(\Delta_h \overline{{\bf u}}_{h}, {\bf u} \times \Delta \pmb{\Psi}\big)|
+ | \big(\Delta_h \overline{{\bf u}}_{h}, {\bf u} \times \Delta \pmb{\Psi}\big) - \big(\Delta {\bf u}, {\bf u} \times \Delta \pmb{\Psi}\big)|
\\
& 
:= A_1+A_2+A_3+A_4+A_5.
\end{align*}
We estimate the first term by the Cauchy-Schwarz inequality and $|\mathbf{u}_h| \leq 1$ as 
\begin{align*}
A_1 \leq C \|\Delta_h \overline{{\bf u}}_{h}\|_h \| \Delta_h \mathcal{I}_h \pmb{\Psi} - \Delta \pmb{\Psi}\|_h.
\end{align*}
Hence, noting Assumption~\ref{ass1} and Lemma~\ref{lem_est1} we deduce that
$\lim_{h\rightarrow 0} \int_0^T A_1 \, \mathrm{d}t = 0.$

To estimate the second term we use (\ref{convdiscnorm}), $|{\bf u}_h| \leq 1$ and $W^{1,\infty}$-stability of the interpolation operator
$$
A_2 \leq C h \|\Delta_h \overline{{\bf u}}_{h}\| \|\nabla (\overline{{\bf u}}_h \times \mathcal{I}_h\Delta\pmb{\Psi})\|
\leq C h \|\Delta_h \overline{{\bf u}}_{h}\|\left(\|\nabla \overline{{\bf u}}_h\| + 1 \right)\|\Delta\pmb{\Psi}\|_{W^{1,\infty}}.
$$
Consequently, noting Lemma~\ref{lem_est1} and  $\|\nabla \overline{{\bf u}}_h\|_{L^\infty(0,T; \mathbf{L}^2)}\leq C$ we get that $\lim_{h\rightarrow 0} \int_0^T A_2\, \mathrm{d}t = 0$ .

Next, Lemma~\ref{lem_est1}, $|\mathbf{u}_h| \leq 1$ and (\ref{est_interpol}) imply that
$$
A_3 \leq Ch^2 \|\Delta_h \overline{{\bf u}}_{h}\| |\Delta \pmb{\Psi}|_{W^{2,2}} \leq C h^2,
$$
and consequently $\lim_{h\rightarrow 0} \int_0^T A_3\, \mathrm{d}t = 0$.

Using the Cauchy-Schwarz inequality we estimate
$$
A_4 \leq \|\Delta_h \overline{{\bf u}}_{h}\| \|\overline{{\bf u}}_{h} - {\bf u}\| \|\Delta\pmb{\Psi} \|_{\mathbf{L}^\infty},
$$
and consequently Lemma~\ref{lem_est1} and (\ref{strongconv}) imply that
$\lim_{h\rightarrow 0} \int_0^T A_4\, \mathrm{d}t = 0.$
For the last term we conclude by the weak convergence (\ref{weakconv_deltah}) that $\lim_{\tau, h\rightarrow 0} \int_0^T A_5\, \mathrm{d}t = 0$.

Collecting the above limits for $A_1$, $\dots$, $A_5$ we conclude that
\begin{equation}\label{conv_lapterm}
\lim_{\tau,h\rightarrow 0}\int_0^T \big(\Delta_h \overline{{\bf u}}_{h}, \overline{{\bf u}}_h \times \Delta_h \mathcal{I}_h \pmb{\Psi}\big)_h\mathrm{d}t
 = \int_0^T \big(\Delta {\bf u}, {\bf u} \times \Delta  \pmb{\Psi}\big) \mathrm{d}t.
\end{equation}

The second term in (\ref{lap_terms}) is estimated as
\begin{align*}
& |\big(\Delta_h \overline{{\bf u}}_{h}, \nabla  \overline{{\bf u}}_h \times \nabla \mathcal{I}_h \pmb{\Psi}\big)
- \big(\Delta {{\bf u}}, \nabla  {{\bf u}} \times \nabla  \pmb{\Psi}\big) |
\leq
 |\big(\Delta_h \overline{{\bf u}}_h, \nabla  \overline{{\bf u}}_h \times \nabla \mathcal{I}_h \pmb{\Psi}\big)
- \big(\Delta_h \overline{{\bf u}}_h, \nabla  \overline{{\bf u}}_h \times \nabla  \pmb{\Psi}\big) |
\\
& \qquad
+|\big(\Delta_h \overline{{\bf u}}_h, \nabla  \overline{{\bf u}}_h \times \nabla  \pmb{\Psi}\big)
    - \big(\Delta_h \overline{{\bf u}}_h, \nabla  {\bf u} \times \nabla  \pmb{\Psi}\big) |
+ |\big(\Delta_h \overline{{\bf u}}_h, \nabla  {\bf u}\times \nabla  \pmb{\Psi}\big)
    - \big(\Delta {\bf u}, \nabla {\bf u} \times \nabla  \pmb{\Psi}\big) |
\\
& := B_1 + B_2 + B_3
\end{align*}
Using Lemma~\ref{lem_est1} along with the bound $\|\nabla {\bf u}_h\|_{L^{\infty}(0,T; {\bf L}^{2})}\leq C$ and (\ref{est_interpol}) we estimate
$$
B_1 \leq \|\Delta_h \overline{{\bf u}}_h\| \|\nabla  \overline{{\bf u}}_h\| \|\mathcal{I}_h \pmb{\Psi} - \pmb{\Psi} \|_{\mathbf{W}^{1,\infty}} \leq C h |\pmb{\Psi} |_{\mathbf{W}^{2,\infty}}.
$$
The second term can be estimate by Lemma~\ref{lem_est1} as
$$
B_2 \leq \|\Delta_h \overline{{\bf u}}_h\| \|\nabla \overline{{\bf u}}_h - \nabla {\bf u}\| \|\pmb{\Psi}\|_{\mathbf{W}^{1,\infty}}.
$$
Consequently, noting the convergences (\ref{weakconv_deltah}), (\ref{conv_grad}) we conclude that
\begin{equation}\label{conv_gradterm}
\lim_{\tau,h\rightarrow 0 }\int_0^T \big(\Delta_h \overline{{\bf u}}_{h}, \nabla  \overline{{\bf u}}_h \times \nabla \mathcal{I}_h \pmb{\Psi}\big)\mathrm{d}t
= \int_0^T \big(\Delta {{\bf u}}, \nabla  {{\bf u}} \times \nabla  \pmb{\Psi}\big)\mathrm{d}t.
\end{equation}

The remaining terms can be treated as in \cite[Theorem 4.1]{num_harm09}. I.e., we have that
$$
\lim_{\tau,h\rightarrow 0} \int_0^T \big(\partial_t {\bf u}_{h}, \partial_t [{\bf u}_h \times \pmb{\Psi}]\big)_h \mathrm{d}t
=\int_0^T \big(\partial_t {\bf u}, \partial_t [{\bf u} \times \pmb{\Psi}]\big) \mathrm{d}t,
$$
and
$$
\lim_{\tau,h\rightarrow 0} \big(\mathbf{V}^0, \mathbf{U}^0\times \pmb{\Psi}(0)\big)_h = \big(\mathbf{v}^0, \mathbf{u}^0\times \pmb{\Psi}(0)\big),
$$
for all $\pmb{\Psi}\in  C^\infty_0([0,T); C^\infty(\Omega;\mathbb{R}^3))$.
Furthermore, owing to (\ref{num_damp}), we get for the first term on the right-hand side of (\ref{schemeh2}) that
$$
\lim_{\tau\rightarrow 0}\int_0^T \big([{\bf v}_{h} - {\bf v}_{h}^+], \partial_t [{\bf u}_h \times \pmb{\Psi}]\big)_h \mathrm{d} t= 0.
$$

Collecting last three eqalities along with (\ref{conv_lapterm}), (\ref{conv_gradterm}) we conclude from (\ref{schemeh2}) that the limit ${\bf u}$ satisfies (\ref{weakform}).
\end{proof}

\section{Higher-dimensional case}\label{sec_stab}
In the higher-dimensional case $d>1$ we employ Lemma~\ref{lem_product} (instead of (\ref{lap_terms}) for $d=1$) and get that
\begin{equation}\label{lap_termshd}
\begin{split}
\lim_{\tau,h\rightarrow0}\big(\Delta_h \overline{{\bf u}}_{h}, \Delta_h \mathcal{I}_h [\overline{{\bf u}}_h \times \pmb{\Psi}]\big)_h
 \leq  &
\lim_{\tau,h\rightarrow0} \left(\big(\Delta_h \overline{{\bf u}}_{h}, \overline{{\bf u}}_h \times \Delta_h \mathcal{I}_h \pmb{\Psi}\big)_h
+  2 \big(\Delta_h \overline{{\bf u}}_{h}, \nabla  \overline{{\bf u}}_h \times \nabla \mathcal{I}_h \pmb{\Psi}\big)\right)
\\
& + \lim_{\tau,h\rightarrow0} C  h \|\nabla \Delta_h \overline{{\bf u}}_h\|\|\nabla \overline{{\bf u}}_h\|\|\nabla \pmb{\Psi}\|_{{\bf L}^\infty}.
\end{split}
\end{equation}
When passing to the limit in the discrete formulation we need to show that the perturbation term
$h \|\nabla \Delta_h \overline{{\bf u}}_h\|\|\nabla \overline{{\bf u}}_h\|\|\nabla \pmb{\Psi}\|_{{\bf L}^\infty}$  disappears for $h\rightarrow 0$.
However we lack an $h$-independent estimate for  $\nabla \Delta_h \overline{{\bf u}}_h$.
Estimating the term $\|\nabla \Delta_h \overline{{\bf u}}_h\|$ by an inverse estimate is not viable since it would introduce a factor $h^{-1}$ which would cancel the factor
$h$ in the last term in (\ref{lap_termshd}).
Under the assumption that $h^{\alpha/2}\|\nabla \Delta_h \overline{{\bf u}}_h\| < C$ for some $0<\alpha<2$  (cf. (\ref{discener_stab}) below) we can control the perturbation term as follows
\begin{equation}\label{conv_pert}
h \int_0^T\|\nabla \Delta_h \overline{{\bf u}}_h\|\|\nabla \overline{{\bf u}}_h\|\|\nabla \pmb{\Psi}\|_{{\bf L}^\infty} \mathrm{d}t\leq
C h^{1-\alpha/2} \|\nabla \pmb{\Psi}\|_{L^\infty(0,T; {\bf L}^\infty)} \rightarrow 0 \text{ for } h\rightarrow 0.
\end{equation}

To guarantee a $h$-independent bound of $\nabla \Delta_h \overline{{\bf u}}_h$ we introduce an artificial stabilization term in the numerical formulation (\ref{scheme}).
The corresponding numerical approximation is defined as follows:
given $\mathbf{U}^{0}$, $\mathbf{V}^{0}$ set $\mathbf{U}^{-1} = \mathbf{U}^{0} - \tau \mathbf{V}^0$
and for $n=0,\dots, N-1$ determine $\mathbf{U}^{n+1}, \mathbf{W}^{n+1}\in {\bf V}_h$ such that for all $\pmb{\Phi}, \pmb{\Psi} \in {\bf V}_h$ there holds
\begin{align}\label{scheme_stab}
\bigl(d_t^2 \mathbf{U}^{n+1}, \pmb{\Phi} \bigr)_h + (\nabla \mathbf{W}^{n+1/2}, \nabla \pmb{\Phi})
{+} h^\alpha\left(\Delta_h \mathbf{W}^{n+1/2},  \Delta_h \pmb{\Phi} \right)_{\rev{h}}
& = \bigl( \lambda_w^{n+1} \mathbf{U}^{n+1/2},  \pmb{\Phi}\bigr)_h \,, 
\\ \nonumber
(\mathbf{W}^{n+1/2}, \pmb{\Psi})_h & = (\nabla \mathbf{U}^{n+1/2}, \nabla\pmb{\Psi}) \,,
\end{align}
where $0<\alpha < 2$ and the discrete Lagrange multiplier $\lambda_w^{n+1} \in V_h$ is defined for $\bz \in \N$ as
\begin{eqnarray}
\label{lambdaw_hd}
&& \lambda_w^{n+1}(\bz) =
\left\{ \begin{array}{ll}
0 & \mbox{ if } \mathbf{U}^{n+1/2}(\bz) = {\bf 0},
\\
\frac{\frac{1}{2}\langle \mathbf{V}^0(\bz), \mathbf{V}^1(\bz)\rangle}{\vert \mathbf{U}^{1/2}(\bz)\vert^2}
+ \frac{\bigl( \nabla \mathbf{W}^{1/2}, \mathbf{U}^{1/2}(\bz) \otimes \nabla \varphi_{\bz}\bigr)}{\beta_{\bz}  \vert \mathbf{U}^{1/2}(\bz)\vert^2}
\\
+ \frac{h^\alpha \bigl( \Delta_h \mathbf{W}^{1/2}, \mathbf{U}^{1/2}(\bz)  \Delta_h \varphi_{\bz} \bigr)_{\rev{h}}}{\beta_{\bz}  \vert \mathbf{U}^{1/2}(\bz)\vert^2}
& \text{if } n=0,\, \mathbf{U}^{1/2}(\bz) \neq {\bf 0},
\\
-\frac{\bigl\langle d_t \mathbf{U}^n(\bz), {d_t}\mathbf{U}^{n+1/2}(\bz)\bigr\rangle}{\vert \mathbf{U}^{n+1/2}(\bz)\vert^2}
+ \frac{\bigl( \nabla \mathbf{W}^{n+1/2}, \mathbf{U}^{n+1/2}(\bz) \otimes \nabla \varphi_{\bz} \bigr)}{\beta_{\bz}  \vert \mathbf{U}^{n+1/2}(\bz)\vert^2}
\\
+ \frac{h^\alpha \bigl( \Delta_h \mathbf{W}^{n+1/2}, \mathbf{U}^{n+1/2}(\bz)  \Delta_h \varphi_{\bz} \bigr)_{\rev{h}}}{\beta_{\bz}  \vert \mathbf{U}^{n+1/2}(\bz)\vert^2}
& \mbox{ else}\, .
\end{array}  \right.
\end{eqnarray}

Setting $\pmb{\Phi} = \mathbf{U}^{n+1} - \mathbf{U}^{n}$, $\pmb{\Psi} =  \Delta_h (\mathbf{U}^{n+1} - \mathbf{U}^{n})$ in (\ref{scheme_stab}) and noting that $\mathbf{W}^n = -\Delta_h \mathbf{U}^n$
we deduce that the following stability estimate holds
\begin{equation}\label{discener_stab}
\max_{n=1,\dots,N}\left[\mathcal{E}_h \bigl({\bf V}^{n}, {\bf U}^{n}\bigr) + \frac{h^\alpha}{2} \|\nabla \mathbf{W}^n \|^2\right]
 \leq \mathcal{E}_h \bigl( {\bf V}^0, {\bf U}^0\bigr) + \frac{h^\alpha}{2} \|\nabla \mathbf{W}^0 \|^2\, ,
\end{equation}

\subsection{Existence}
The existence of $\mathbf{U}^{n+1}$, $\mathbf{W}^{n+1} = -\Delta_h \mathbf{U}^{n+1}$ which solve  (\ref{scheme_stab}) follow as in Lemma~\ref{lem_exist}
under the assumption that $\tau \leq C h^{3-\alpha/2}$ is sufficiently small.
The additional term (\ref{lambdaw_hd}) (which imposes the stronger assumption $\tau \leq C h^{3-\alpha/2}$) can be estimated analogously to (\ref{exest2}) as
\begin{align*}
\frac{h^\alpha \bigl( \Delta_h \mathbf{W}^{n+1/2}, \mathbf{U}^{n+1/2}(\bz)  \Delta_h \varphi_{\bz} \bigr)_{\rev{h}}}{\beta_{\bz}  \vert \mathbf{U}^{n+1/2}(\bz)\vert^2}
& \leq C \frac{h^\alpha}{h^2}\|\Delta_h \mathbf{W}^{n+1/2}\|_h \|\mathbf{U}^{n+1/2}\|_h \leq C \frac{h^\alpha}{2h^3}\|\nabla \mathbf{W}^{n+1/2}\| \|\mathbf{U}^{n+1/2}\|_h
\\
& \leq \frac{h^\alpha}{2}\|\nabla \mathbf{W}^{n+1/2}\|^2 +  C h^{\alpha-6}\|\mathbf{U}^{n+1/2}\|_h^2.
\end{align*}
where we used the inverse estimate (\ref{inv_disclap}) and that $\|\Delta_h \varphi_{\bz}\|_{L^\infty}\leq C h^{-2}$ .

Moreover to show (\ref{dt_small}) we estimate the additional term
\begin{align*}
\sum_{{\bf z}\in\mathcal{N}_h} & \beta_{{\bf z}}\frac{h^\alpha \bigl( \Delta_h \mathbf{W}^{n+1/2}, \mathbf{U}^{n+1/2}(\bz)  \Delta_h \varphi_{\bz} \bigr)_{\rev{h}}}{\beta_{\bz}  \vert \mathbf{U}^{n+1/2}(\bz)\vert^2}
d_t \mathbf{U}^{n+1}(\bz)
 \leq Ch^{\alpha-3}\|\nabla \mathbf{W}^{n+1/2}\|\|d_t \mathbf{U}^{n+1}\|_h
\\
& \leq C{h^{\alpha/2-3}} h^{\alpha}\|\nabla \mathbf{W}^{n+1/2}\|^2 + C {h^{\alpha/2 -3}}\|d_t \mathbf{U}^{n+1}\|_h^2 ,
\end{align*}
which can be absorbed in the corresponding terms for sufficiently small $\tau \leq C{h^{3- \alpha/2}}$

\subsection{Convergence}
Under the additional assumption that $h^\alpha\|\nabla \mathbf{W}^0 \|^2 \leq C$
the convergence of the numerical approximation (\ref{scheme_stab}) follows as in Theorem~\ref{thm_converg}
with minor modifications:
instead of (\ref{lap_terms}) we consider (\ref{lap_termshd}) and the perturbation term disappears for $h\rightarrow 0$
thanks to (\ref{conv_pert});
moreover thanks to (\ref{discener_stab}) the stabilization term vanishes in the limit, i.e.,
$h^\alpha\left(\Delta_h \mathbf{W}^{n+1/2},  \Delta_h \pmb{\Phi} \right)_{\rev{h}} \leq h^\alpha \|\Delta_h \mathbf{W}^{n+1/2}\| \| \Delta_h \pmb{\Phi}\| \rightarrow 0$ for $h\rightarrow 0$. The remainder of the proof is the same as in the case $d=1$.

\rev{
\begin{remark}\label{rem_scal}
For practical purposes it is convenient to consider a scaled version of the stabilized scheme \eqref{scheme_stab} where
the stabilization term is scaled by an additional constant $C_{{\tt stab}} > 0$, i.e., one employs the stabilization term $C_{{\tt stab}} h^\alpha\left(\Delta_h \mathbf{W}^{n+1/2},  \Delta_h \pmb{\Phi} \right)_{h}$ in \eqref{scheme_stab}.
It is obvious that that the convergence proof also holds for the scaled version of the stabilised scheme with arbitrarily small ($\tau$, $h$-independent) constant $C_{{\tt stab}}$.
\end{remark}
}

\section{Conservative scheme}\label{sec_cons}

We present an alternative approximation scheme that preserves the discrete energy, i.e, the scheme satisfies a (discrete) energy equality (\ref{discener2}) below.

Let $\mathbf{U}^0, \mathbf{V}^{0}\in\mathbf{V}_h$, $\langle\mathbf{V}^{0}({\bf z}), \mathbf{U}^{0}({\bf z}) \rangle =0$, $\mathbf{U}^{-1}=\mathbf{U}^0-\tau\mathbf{V}^0$, for $n\geq1$
define $\mathbf{V}^{n} = d_t \mathbf{U}^{n}$ and
$\mathbf{U}^{n^*+\frac{1}{2}} = \frac{1}{2}\left(\mathbf{U}^{n+1} + \mathbf{U}^{n-1}\right)$, $\mathbf{W}^{n^*+\frac{1}{2}} = \frac{1}{2}\left(\mathbf{W}^{n+1} + \mathbf{W}^{n-1}\right)$.
For $n=0,\dots, N-1$ determine $\mathbf{U}^{n+1}$, $\mathbf{W}^{n+1}$ as the solution of
\begin{align}\label{scheme_cons}
\bigl(d_t\mathbf{V}^{n+1}, \pmb{\Phi} \bigr)_h + (\nabla \mathbf{W}^{n^*+\frac{1}{2}}, \nabla
\pmb{\Phi}) & = \bigl( \widetilde{\lambda}_w^{n+1} \mathbf{U}^{n^*+\frac{1}{2}},  \pmb{\Phi}\bigr)_h \,,
\\ \nonumber
(\mathbf{W}^{n^*+\frac{1}{2}}, \pmb{\Psi})_h & = (\nabla \mathbf{U}^{n^*+\frac{1}{2}}, \nabla\pmb{\Psi}) \,,
\end{align}
with
$$
\widetilde{\lambda}_w^{n+1}(\bz) =
\left\{ \begin{array}{ll}
0 & \text{if } \mathbf{U}^{n^*+1/2}(\bz) = {\bf 0},
\\
-\frac{\bigl\langle \mathbf{V}^0(\bz), \mathbf{V}^{1}(\bz)\bigr\rangle - \frac{1}{2}|\mathbf{V}^0(\bz)|^2}{\vert \mathbf{U}^{0^*+\frac12}(\bz)\vert^2}
+ \frac{\bigl( \nabla \mathbf{W}^{0^*+\frac12}, \mathbf{U}^{0^*+\frac12}(\bz) \otimes \nabla \varphi_{\bz}
\bigr)}{\beta_{\bz}  \vert \mathbf{U}^{0^*+\frac12}(\bz)\vert^2}
& \text{if } n=0,\, \mathbf{U}^{0^*+\frac12}(\bz) \neq {\bf 0}
\\
-\frac{\bigl\langle  \mathbf{V}^n(\bz), \mathbf{V}^{n+1}(\bz)\bigr\rangle}{\vert \mathbf{U}^{n^*+\frac12}(\bz)\vert^2}
+ \frac{\bigl( \nabla \mathbf{W}^{n^*+\frac12}, \mathbf{U}^{n^*+\frac12}(\bz) \otimes \nabla \varphi_{\bz}
\bigr)}{\beta_{\bz}  \vert \mathbf{U}^{n^*+\frac12}(\bz)\vert^2}
& \text{else},
\end{array}  \right.
$$
where the modification for $n=0$, accounts for the fact that possibly $|{\bf U}^{-1}|\neq 1$ if $\mathbf{V}_0 \neq {\bf 0}$, cf. \cite{num_harm09}, \cite{swave}.

Setting $\pmb{\Phi} = \mathbf{U}^{n^*+\frac12}(\bz)\varphi_{\bz}$ in (\ref{scheme_cons}) then implies that $\vert {\bf U}^{n+1}(\bz)\vert^2 = \vert {\bf U}^n(\bz)\vert^2$ $\forall \bz\in\mathcal{N}_h$, $n=0,\dots,N-1$.

Setting $\pmb{\Phi} = \mathbf{U}^{n+1} - \mathbf{U}^{n-1} = \mathbf{V}^{n+1} + \mathbf{V}^{n}$, $\pmb{\Psi} =  \Delta_h (\mathbf{U}^{n+1} - \mathbf{U}^{n-1})$ in (\ref{scheme_cons})
implies that above scheme conserves the discrete energy
$\widetilde{\mathcal{E}}_h \bigl({\bf V}^{n}, {\bf U}^{n}, {\bf U}^{n-1}\bigr) :=
\frac{1}{2}\left(\|\mathbf{V}^{n}\|_h^2 + \frac{1}{2}\left[ \|\Delta_h \mathbf{U}^{n}\|^2 + \|\Delta_h \mathbf{U}^{n-1}\|_h^2\right]\right)$,
i.e.,  for $n=1,\dots,N$ the following equality holds:
\begin{equation}\label{discener2}
\widetilde{\mathcal{E}}_h \bigl({\bf V}^{n}, {\bf U}^{n}\bigr)
 = \frac{1}{2}\left(\|\mathbf{V}^{0}\|_h^2 + \frac{\tau}{2} \left(\widetilde{\lambda}_w^{1}\mathbf{V}^{0}, \mathbf{V}^{0}\right)_h
   + \frac{1}{2}\left[ \|\Delta_h \mathbf{U}^{0}\|^2 + \|\Delta_h \mathbf{U}^{-1}\|_h^2\right]\right) .
\end{equation}

Compared to (\ref{discener}) the numerical damping term
$\frac{\tau^2}{2} \sum_{n=1}^N \Vert d_t {\bf V}^{n}\Vert^2_h$ is not present in (\ref{discener2}) and the energy of the initial data is exactly preserved at all times.
Due to the lack of numerical damping term $\frac{\tau^2}{2} \sum_{n=1}^N \Vert d_t {\bf V}^{n}\Vert^2_h$ the convergence of the conservative scheme is not clear.
In particular we are not able to show that the right hand side in the (counterpart of) Lemma~\ref{lem_reform} vanishes for $\tau \rightarrow 0$.

\section{Numerical experiments}\label{sec_exp}

\rev{In this section we perform numerical simulations using the scheme (\ref{scheme}), its stabilized variant (\ref{scheme_stab}) and the energy conservative scheme (\ref{scheme_cons}).
The results for all three schemes were similar on finite time intervals.
In the long-time limit the results will inevitably diverge (for fixed discretization parameters), since
the numerical solutions of (\ref{scheme}) (as well as of  (\ref{scheme_stab})) will eventually become stationary,
due to the effects of the numerical dissipation.}

\subsection{Solution of the nonlinear systems}\label{sec_fixpnt}
For every $n \geq 1$, we solve the nonlinear problem (\ref{scheme})-(\ref{lambdaw})
by a fixed point algorithm with projection, the corresponding modifications of the algorithm for (\ref{scheme_stab}),  (\ref{scheme_cons}) are straightforward).
Given a tolerance $\epsilon > 0$ we proceed
for $n=1,\dots,N$ as follows:
set ${\bf u}^{0}= {\bf u}^{0}_{*} = {\bf U}^{n}$, ${\bf w}^{0}= {\bf w}^{0}_{*} = -\Delta_h {\bf U}^{n}$ and iterate for $l=0,\dots$
\begin{enumerate}
\item
\begin{itemize}
\item[(i)]
Find ${\bf u}^{l+1}, \mathbf{w}^{l+1} = -\Delta_h {\bf u}^{l+1}  \in {\bf V}_h$ such that for all $\pmb{\Phi} \in {\bf V}_h$
$$
\begin{array}{rcl}
\frac{1}{\tau}\Bigl( {\bf u}^{l+1}, \pmb{\Phi}\Bigr)_h +
\frac{\tau}{2} \Bigl(\nabla {\bf w}^{l+1}, \nabla \pmb{\Phi}\Bigr)
- \frac{\tau}{2}\Bigl(\lambda^{l}_{w*} {\bf u}^{l+1}, \pmb{\Phi} \Bigr)_h
& = &
\frac{1}{\tau}\Bigl( 2{\bf U}^{n}- {\bf U}^{n-1} , \pmb{\Phi}\Bigr)_h
- \frac{\tau}{2} \Bigl(\nabla {\bf W}^{n}, \nabla \pmb{\Phi}\Bigr)\\
& & +
\frac{\tau}{2} \Bigl({\lambda}^{l}_{w*}  {\bf U}^{n}, \pmb{\Phi} \Bigr)_h\, .
\qquad
\end{array}
$$
\item[(ii)]
For all ${\bf z}\in \mathcal{N}_{h}$ set
\begin{align*}
{\bf u}^{l+1}_{*}({\bf z}) & = \frac{{\bf u}^{l+1}({\bf z})}{|{\bf u}^{l+1}({\bf z})|},
\\
{\bf w}^{l+1}_{*}({\bf z}) & = -\Delta_h {\bf u}^{l+1}_{*}({\bf z})
\, .
\end{align*}
\end{itemize}
\item
Stop the iterations once
$$
\max_{{\bf z}\in\mathcal{N}_{h}}| {\bf u}^{l+1}({\bf z}) - {\bf u}_{*}^{l}({\bf z}) | \leq \epsilon \,.
$$
\item
Set ${\bf U}^{n+1} = {\bf u}_{*}^{l+1}$ and proceed to the next time level.
\end{enumerate}
The Lagrange multiplier $\lambda_{w*}^{l} $ is a counterpart of (\ref{lambdaw}) which replaces the unknown value ${\bf U}^{n+1}$, ${\bf W}^{n+1}$ by the respective solutions ${\bf u}^{l}_{*}$, ${\bf w}^{l}_{*}$ which are known from the previous
iteration of the nonlinear solver, i.e., we denote $\mathbf{U}^{n+1/2}_* = \frac{1}{2}\big({\bf u}_{*}^{l+1} + {\bf U}^n\big)$, $d_t \mathbf{U}^{n+1}_* = \frac{1}{\tau}\big({\bf u}_{*}^{l+1} - {\bf U}^n\big)$,
and set
\begin{eqnarray}
&& \lambda_{w*}^{l}(\bz) =
\left\{ \begin{array}{ll}
0 & \mbox{if } \mathbf{U}_*^{n+1/2}(\bz) = {\bf 0}\,,
\\
\frac{\frac{1}{2}\langle \mathbf{V}^0(\bz), d_t \mathbf{U}^1_*(\bz) \rangle}{\vert \mathbf{U}_*^{1/2}(\bz)\vert^2}
+ \frac{\bigl( \nabla \mathbf{W}_*^{1/2}, \mathbf{U}_*^{1/2}(\bz) \otimes \nabla \varphi_{\bz}\bigr)}{\beta_{\bz}  \vert \mathbf{U}_*^{1/2}(\bz)\vert^2}
& \text{if } n=0,\, \mathbf{U}_*^{1/2}(\bz) \neq {\bf 0},
\\
-\frac{\bigl\langle d_t \mathbf{U}^n(\bz), {d_t}\mathbf{U}_*^{n+1/2}(\bz)\bigr\rangle}{\vert \mathbf{U}_*^{n+1/2}(\bz)\vert^2}
+ \frac{\bigl( \nabla \mathbf{W}_*^{n+1/2}, \mathbf{U}_*^{n+1/2}(\bz) \otimes \nabla \varphi_{\bz}\bigr)}{\beta_{\bz}  \vert \mathbf{U}_*^{n+1/2}(\bz)\vert^2}
& \mbox{else}\, .
\end{array}  \right.
\end{eqnarray}
We choose the stopping tolerance $\epsilon = 10^{-8}$ in all experiments below.

\begin{remark}[Projection step]
The use of an additional projection step in the fixed-point iteration was introduced in \cite{num_move} for the related second-order problems.
Without the projection step (step $(1)-(ii)$ in the fixed-point algorithm above) the fixed-point algorithm requires to choose $\tau =\mathcal{O}(h^2)$ to guarantee convergence.
Similarly as in \cite{num_move}, the algorithm with the projection step requires a weaker time-step restriction $\tau =\mathcal{O}(h)$ to achieve convergence.
Compared to the second order harmonic wave map equation \cite{num_move}, the bi-harmonic problem requires roughly $100$-times smaller time-step
to similar convergence speed of the fixed-point algorithm, however, we also observe a faster evolution of the solution than in the case of second order problems.
\end{remark}

\subsection{Numerical experiments}

We consider the initial condition on $\Omega = \left(-1,1\right)^2$
\begin{eqnarray}
&&{\bf u}_0({\bf x}) =
\left\{
\begin{array}{ll}
(4{\bf x} A, A^{2} - 4|{\bf x}|^{2})/(A^{2} + 4|{\bf x}|^{2})
      & \mbox{for } |{\bf x}| \le 1/2,
\\
(0, 0, -1)
      & \mbox{for } |{\bf x}| \geq 1/2,
\\
\end{array}\right.
\end{eqnarray}
where $A = (1- 4|\x|)^{4}$.
In the case of second order problems, i.e., the harmonic heat flow and wave map equations, respectively,
the evolution starting from the above initial condition leads to a (discrete) finite-time blow-up, cf., \cite{num_move} and the references therein.

We set $\mathbf{U}^0 = \mathcal{I}_h {\bf u}_0$, $\mathbf{W}^0 = -\Delta_h \mathbf{U}^0$ and $\mathbf{V}^0=\mathbf{0}$.
We compute the problem for different mesh sizes $h=2^{k}$  with corresponding time-step steps $\tau = 2^{6-k} \times 10^{-4}$ for $k=5,6,7,8$.
In Figure~\ref{fig_blowup} we display the evolution of the computed energies as well as of the gradient $\|\nabla \mathbf{U}^n\|_{\mathbf{L}^\infty}$;
\rev{for comparison we also display the (conserved) energy (\ref{discener2}) for $k=5,6,7$ and the gradient for $k=7$ (the corresponding graphs for $k=5,6$ were qualitatively similar)
of the solution computed with the conservative scheme (\ref{scheme_cons})}.
Compared to the case of wave map into spheres, cf. for instance \cite[Example 4.3]{num_move}, we do not observe formation of singularities (discrete blow-up),
i.e., the gradient of the numerical solution remains bounded independently of the mesh size. Snapshots of the solution computed with $h=2^{-5}$ at different times are displayed in
Figure~\ref{fig_uh} (the figure is colored by the magnitude of the $z$-component of the solution).
\begin{figure}[htp!]
\includegraphics[width=0.45\textwidth]{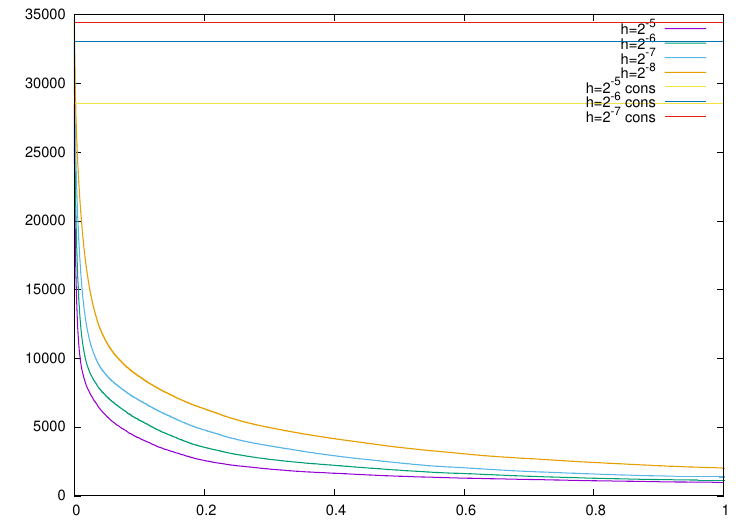}
\includegraphics[width=0.45\textwidth]{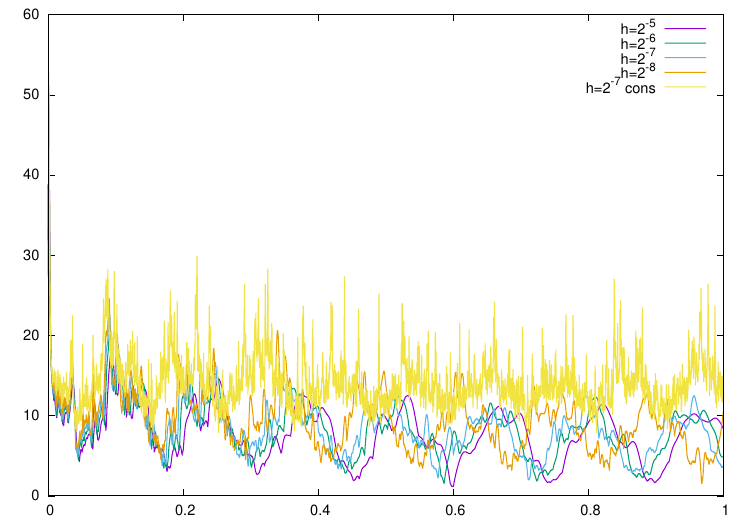}
\caption{Evolution of the discrete energy (left) and of the gradient $\|\nabla \mathbf{U}^n\|_{\mathbf{L}^\infty}$ (right) for $h=2^{k}$, $k=5,6,7,8$.}
\label{fig_blowup}
\end{figure}

\begin{figure}[htp!]
\includegraphics[width=0.32\textwidth]{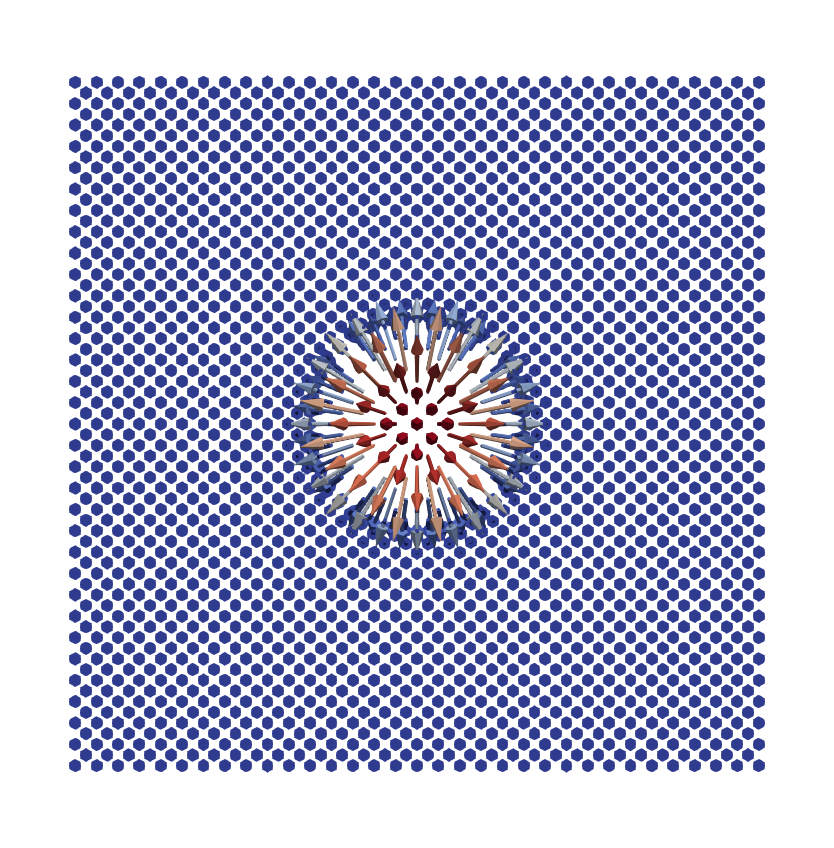}
\includegraphics[width=0.32\textwidth]{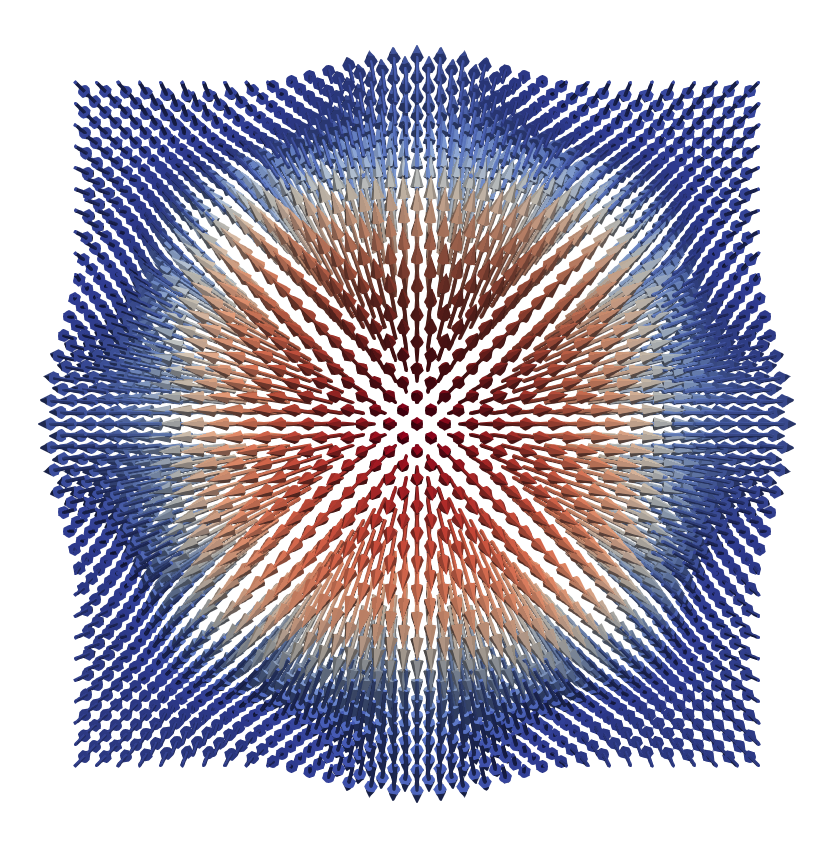}
\includegraphics[width=0.32\textwidth]{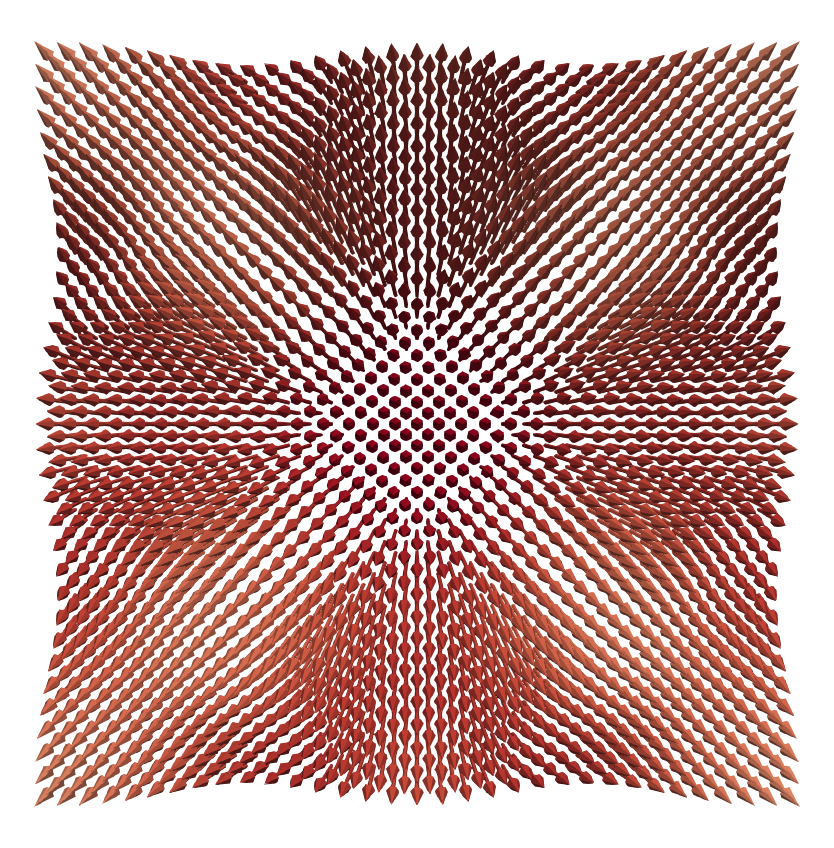}
\includegraphics[width=0.32\textwidth]{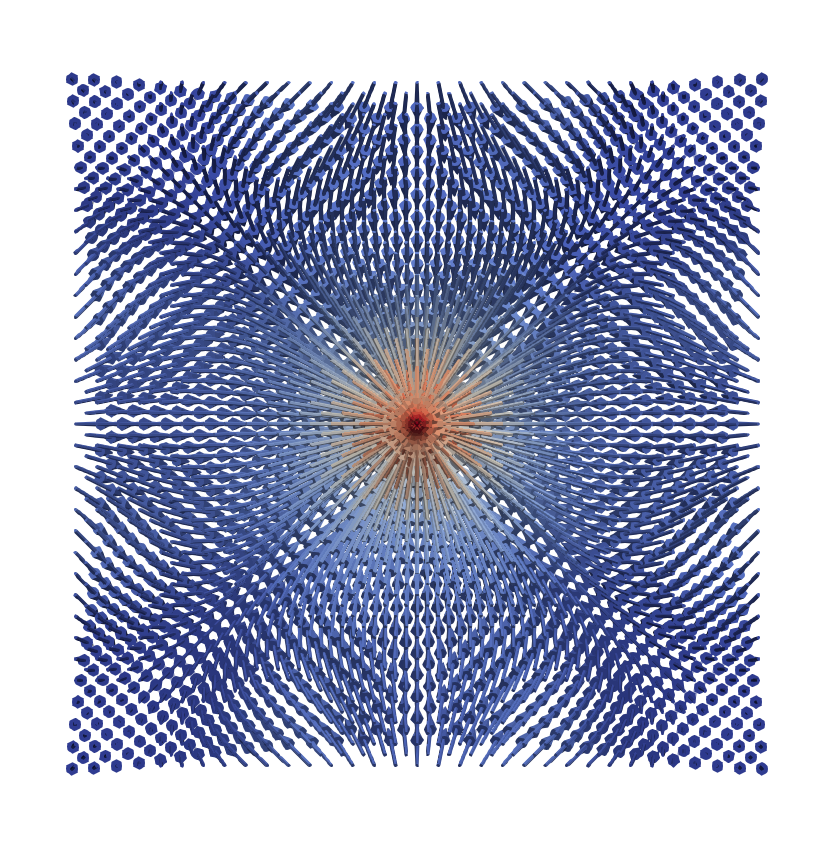}
\includegraphics[width=0.32\textwidth]{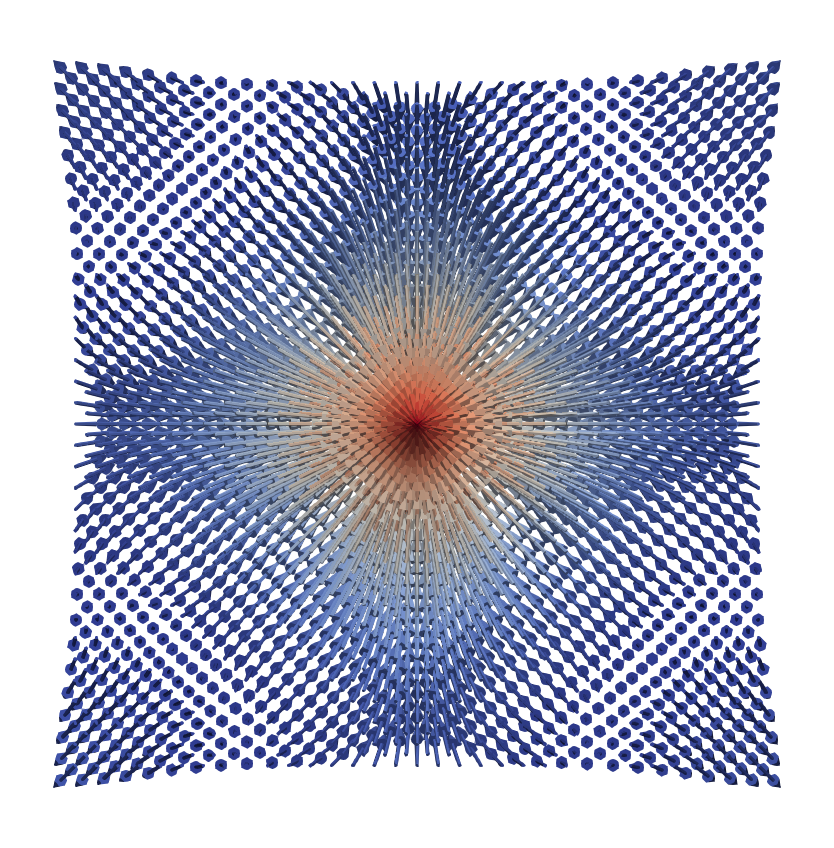}
\includegraphics[width=0.32\textwidth]{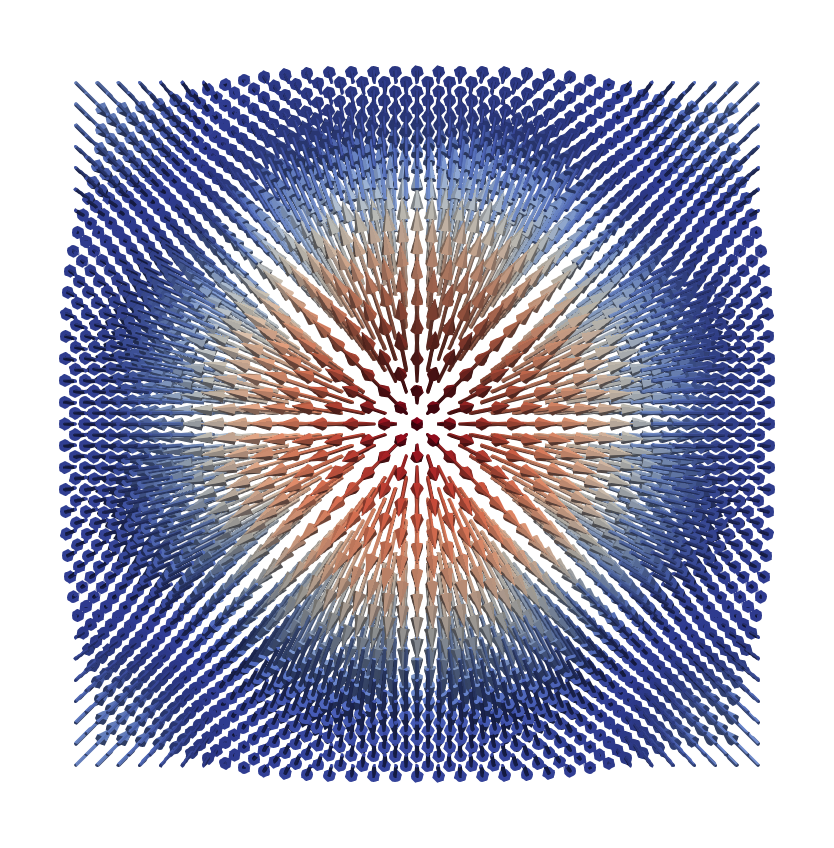}
\caption{Numerical solution computed with $h=2^{-5}$ at time $t=0,0.02,0.045,0.095,0.13,0.24$ (from left to right, top to bottom).}
\label{fig_uh}
\end{figure}

\rev{
Next, we study the effect of the stabilisation term. We consider the same setup as above except for a smaller time step $\tau = 2^{4-k} \times 10^{-4}$.
In Figure~\ref{fig_stab} we compare the results computed with the scaled variant of the stabilised scheme (\ref{scheme_stab}) (cf. Remark~\ref{rem_scal}) with $\alpha=2$ and $C_{\tt stab} = 0.01$
and the scheme without stabilisation, i.e., with $C_{\tt stab} = 0$. The evolution of the discrete energy $\mathcal{E}_h\bigl({\bf V}^{n}, {\bf U}^{n}\bigr)$  and of the energy of the stabilized scheme
$\mathcal{E}_h\bigl({\bf V}^{n}, {\bf U}^{n}\bigr) + \frac{C_{\tt stab} h^\alpha}{2}\|\nabla \mathbf{W}^n\|^2$ (cf. (\ref{discener_stab}))
is displayed in Figure~\ref{fig_stab} on the left and the evolution of the evolution of $\|\nabla \mathbf{W}^n \|^2$ is displayed on the right (we cut-off the initially large values of $\|\nabla \mathbf{W}^n \|^2$ and consider a shorter time interval for better comparison).
Both schemes produce very similar results. In particular, we observe qualitatively similar evolution of $\nabla \mathbf{W}^{n}$ which indicates that
the effect of stabilisation is negligible in the considered setting.
}
\begin{figure}[htp!]
\includegraphics[width=0.45\textwidth]{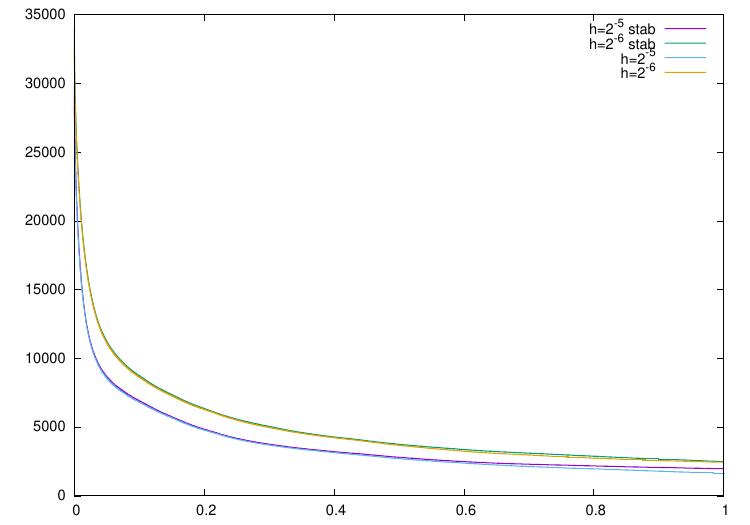}
\includegraphics[width=0.45\textwidth]{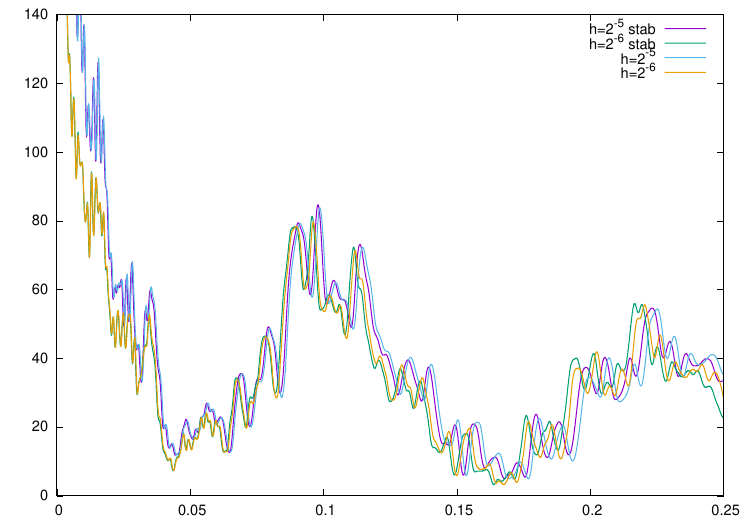}
\caption{Evolution of the respective discrete energies $\mathcal{E}_h$ and  $\mathcal{E}_{h, {\tt stab}}$ (left) and of $\frac{C_{\tt stab} h^2}{2}\|\nabla \mathbf{W}^n\|^2$ (right) for $h=2^{k}$, $k=5,6$.}
\label{fig_stab}
\end{figure}

\rev{In the final experiment we consider a problem with singular initial condition
that goes beyond the theoretical part of this paper: we set
${\bf u}_0(\mathbf{x}) = (1,0,0,)$ for $\mathbf{x} \in \Omega = \left(-1,1\right)^2$ and ${\bf u}_0(\mathbf{x}) = \frac{\mathbf{x}}{|\mathbf{x}|}$ for $\mathbf{x} \in \partial \Omega$
along with a Dirichlet boundary condition ${\bf u}(t)|_{\partial\Omega} = {\bf u}_0|_{\partial\Omega}$ and the homogeneous Neumann boundary condition for ${\bf w}$.
The solution in this setting takes the form ${\bf u} = (u_1, u_2, 0)$, i.e., the third component of the solution is zero.
Similarly as in the previous experiment we observe no significant differences between the results obtained with and without stabilization (despite the fact that the value $h^2\|\nabla \mathbf{W}^0\|^2$ scales with $h^{-3}$), see Figure~\ref{fig_stab_m02}. Note, that due to the singularity of the initial condition, the initial energy is infinite
and convergence of the conservative scheme (\ref{scheme_cons}) is not to be expected in this setting, the corresponding results are therefore not presented.
The evolution of the numerical solution in displayed in Figure~\ref{fig_uh_stab2}. Starting from the given initial condition a singularity quickly forms on the left side of the domain
and the solution then ''oscillates'' around this singularity. The formation of the singularity
(which is an analogue of singularities exhibited by harmonic second order problems, cf. \cite[Example 4.4]{num_move}) is a consequence of the fact that,
the solution remains in the $xy$-plane due to the choice of the initial condition.
Consequently, we also observe that the $L^\infty$-norm of the gradient of the discrete solution scales with the spatial mesh size as $h^{-1}$, see Figure~\ref{fig_stab_m02} (right).
\begin{figure}[htp!]
\includegraphics[width=0.32\textwidth]{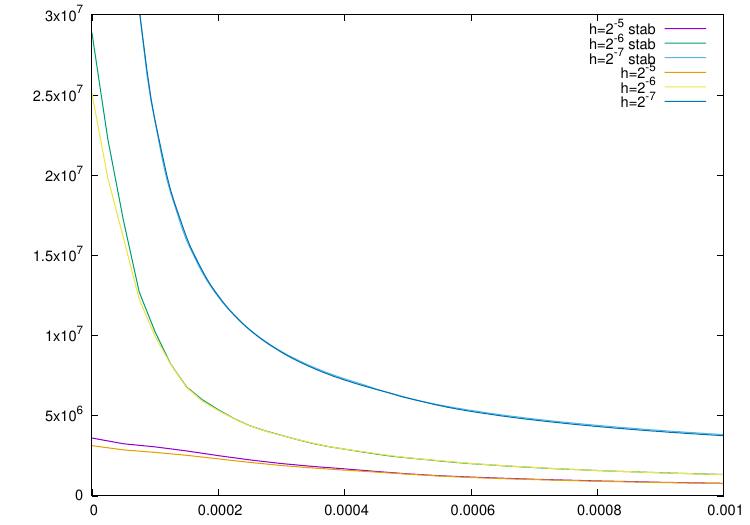}
\includegraphics[width=0.32\textwidth]{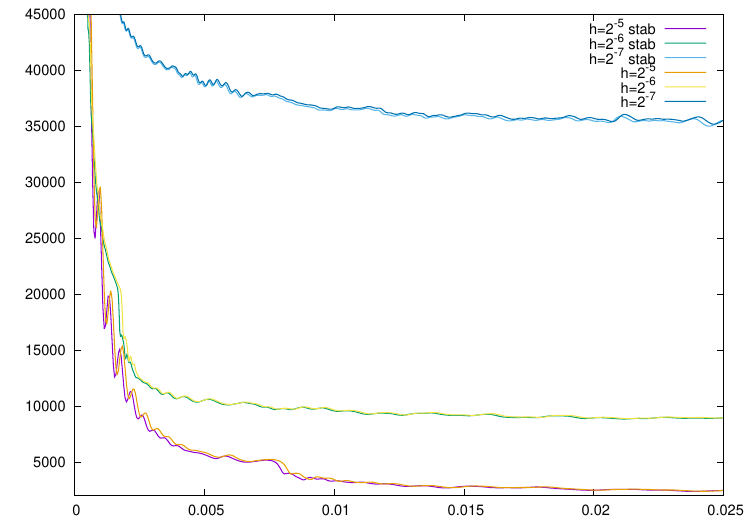}
\includegraphics[width=0.32\textwidth]{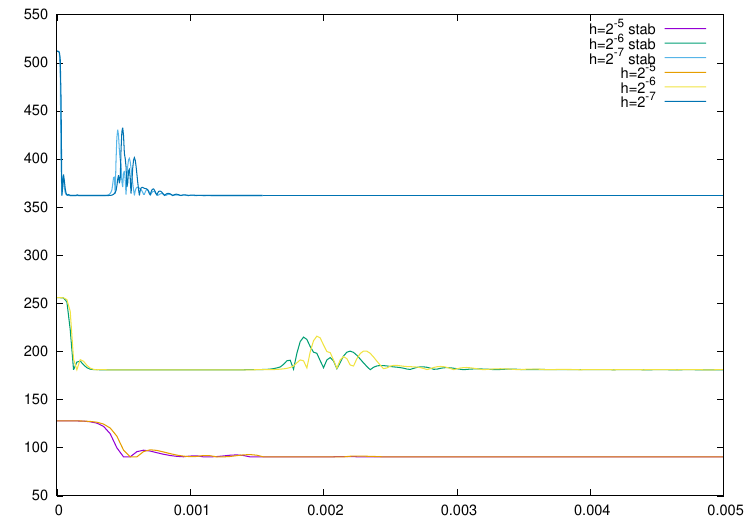}
\caption{Evolution of the respective discrete energies $\mathcal{E}_h$ and  $\mathcal{E}_{h, {\tt stab}}$ (left) and of $\frac{C_{\tt stab} h^2}{2}\|\nabla \mathbf{W}^n\|^2$ (middle)
and the evolution of $\|\nabla \mathbf{U}^n\|_{\mathbf{L}^\infty}$ (right) for $h=2^{k}$, $k=5,6,7$.}
\label{fig_stab_m02}
\end{figure}
\begin{figure}[htp!]
\includegraphics[width=0.32\textwidth]{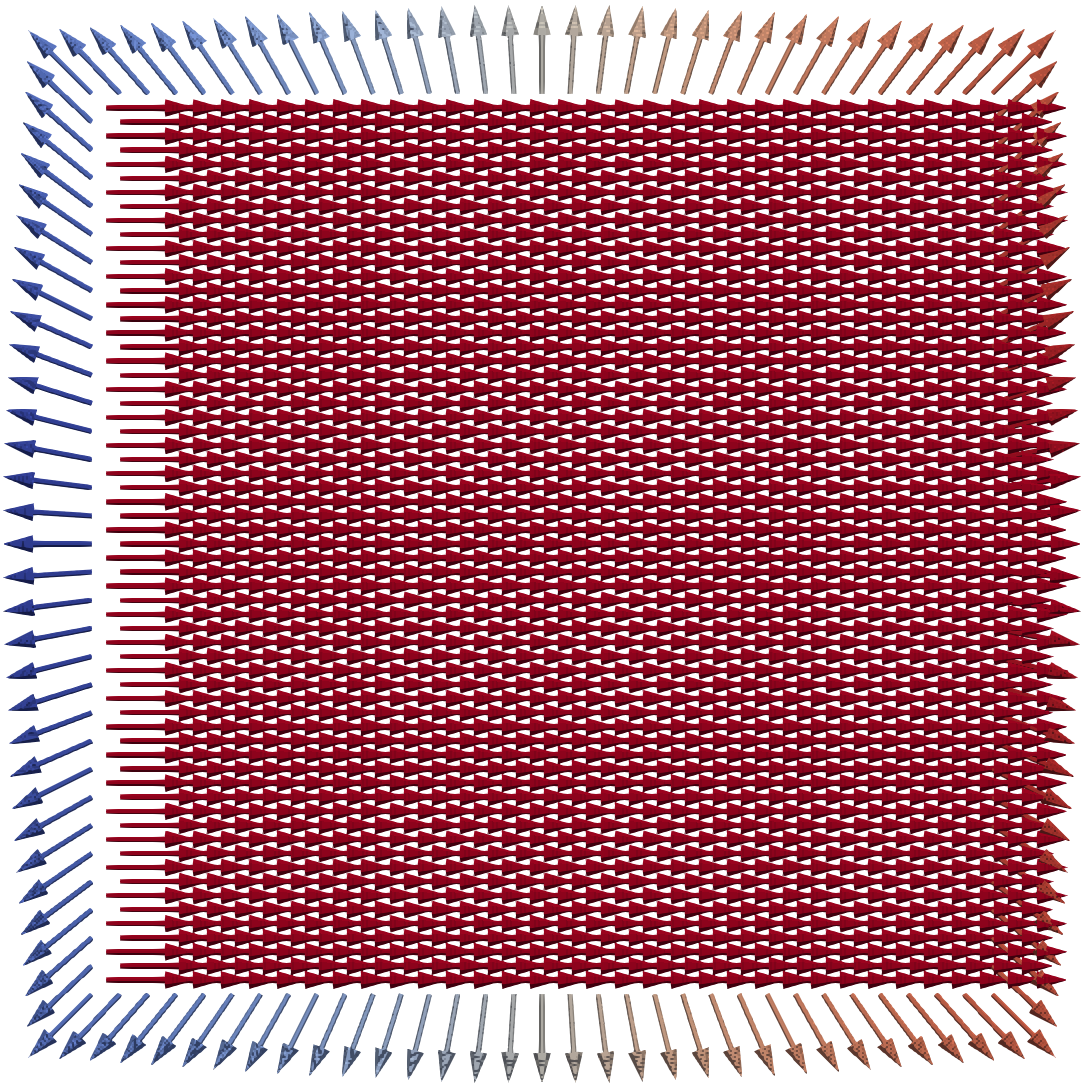}
\includegraphics[width=0.32\textwidth]{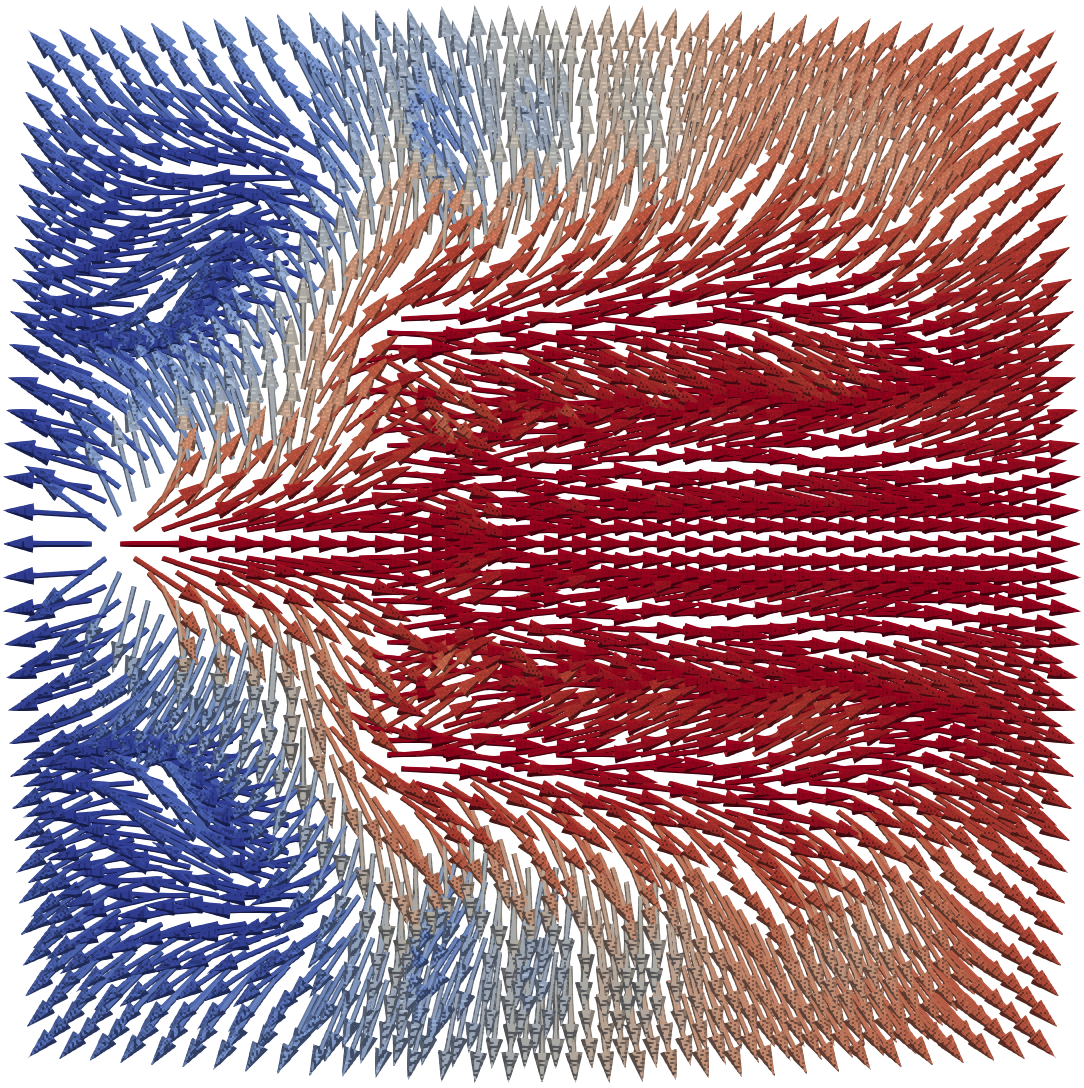}
\includegraphics[width=0.32\textwidth]{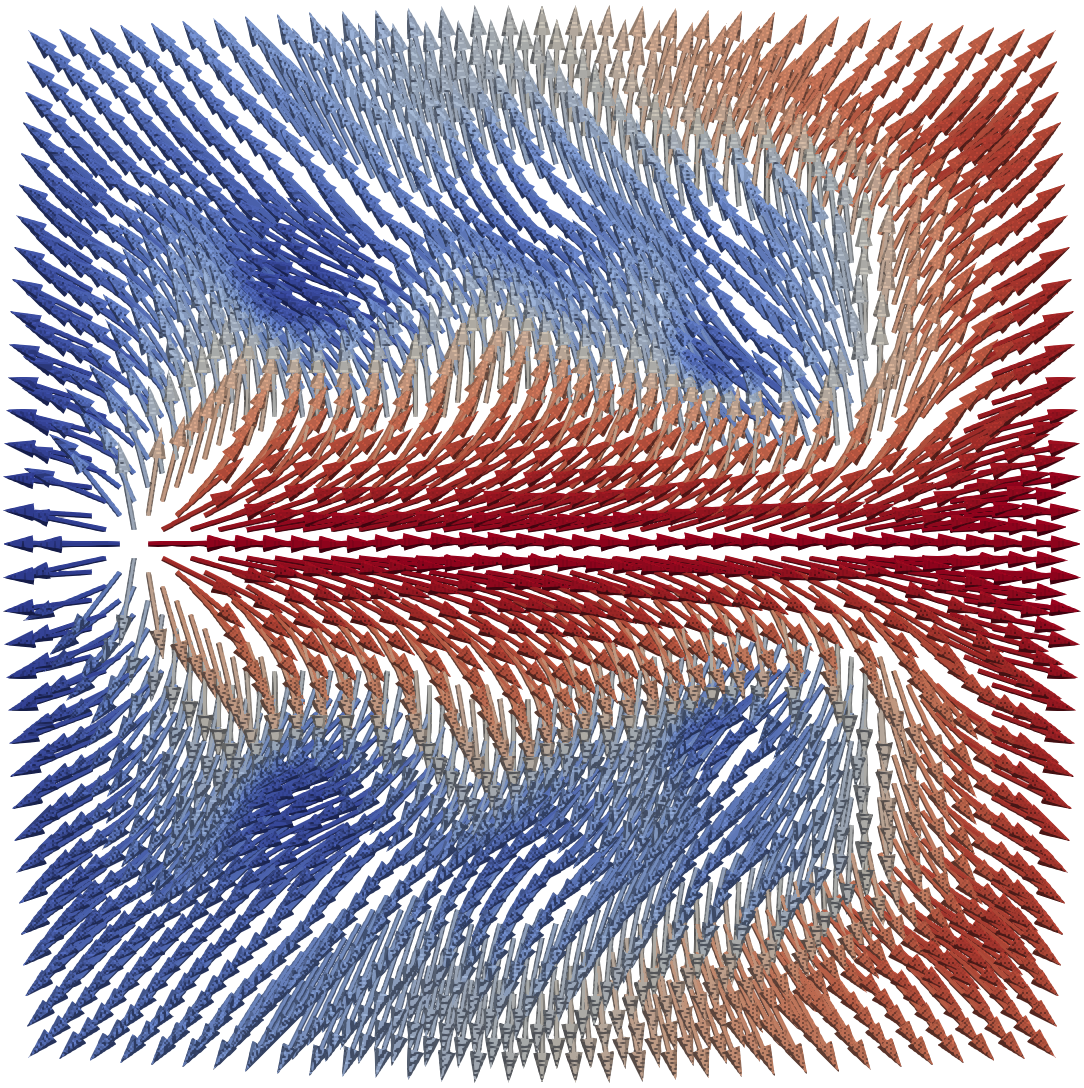}
\includegraphics[width=0.32\textwidth]{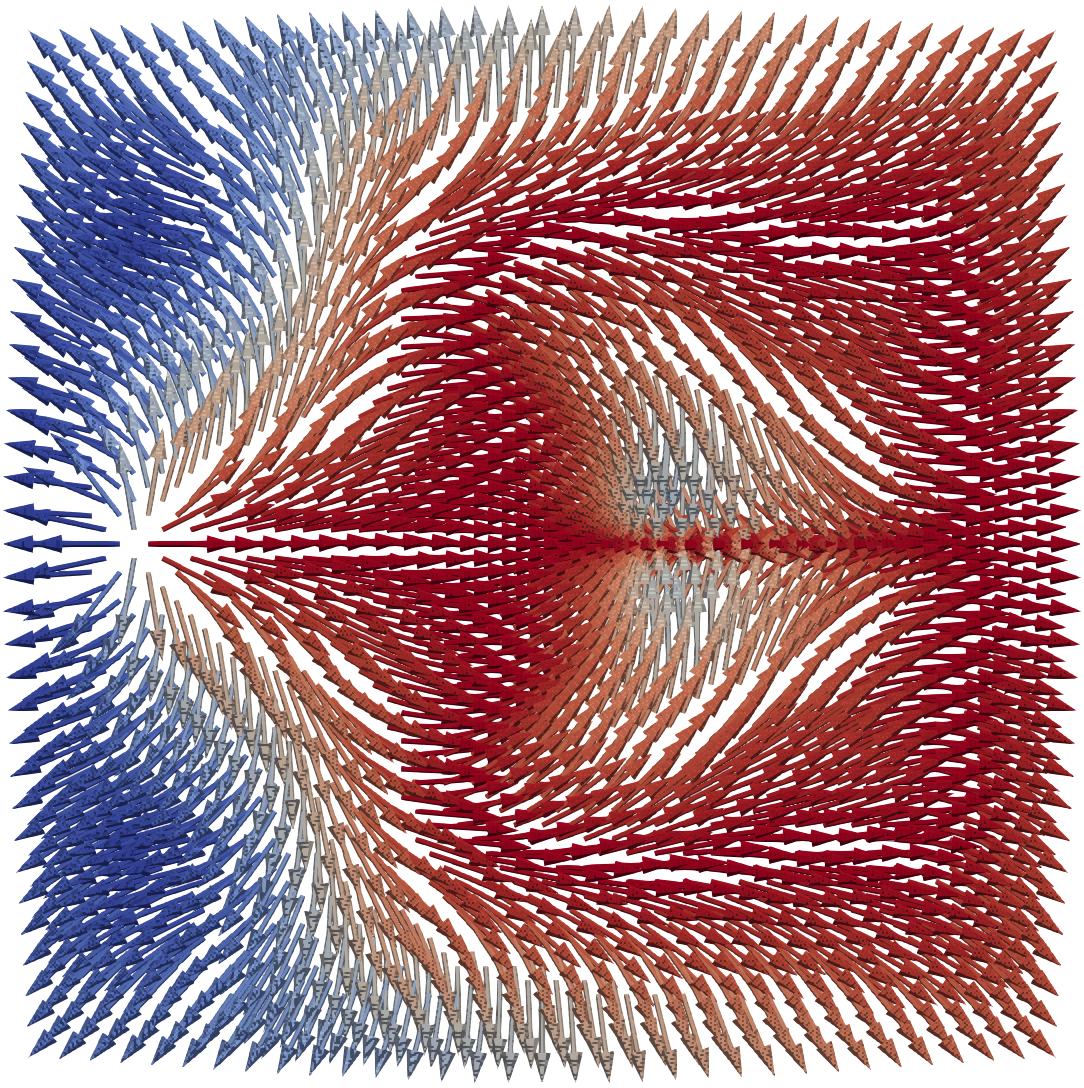}
\includegraphics[width=0.32\textwidth]{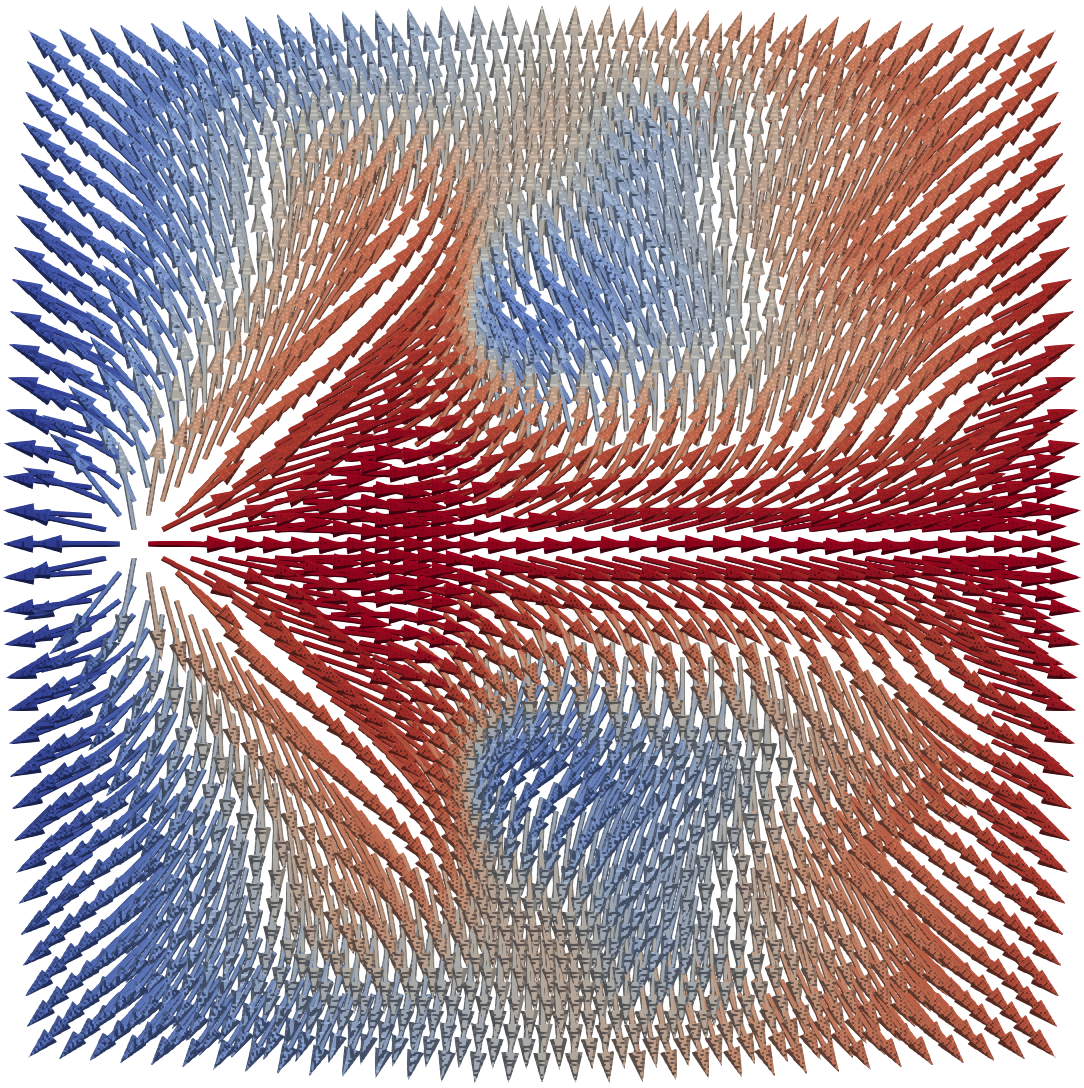}
\includegraphics[width=0.32\textwidth]{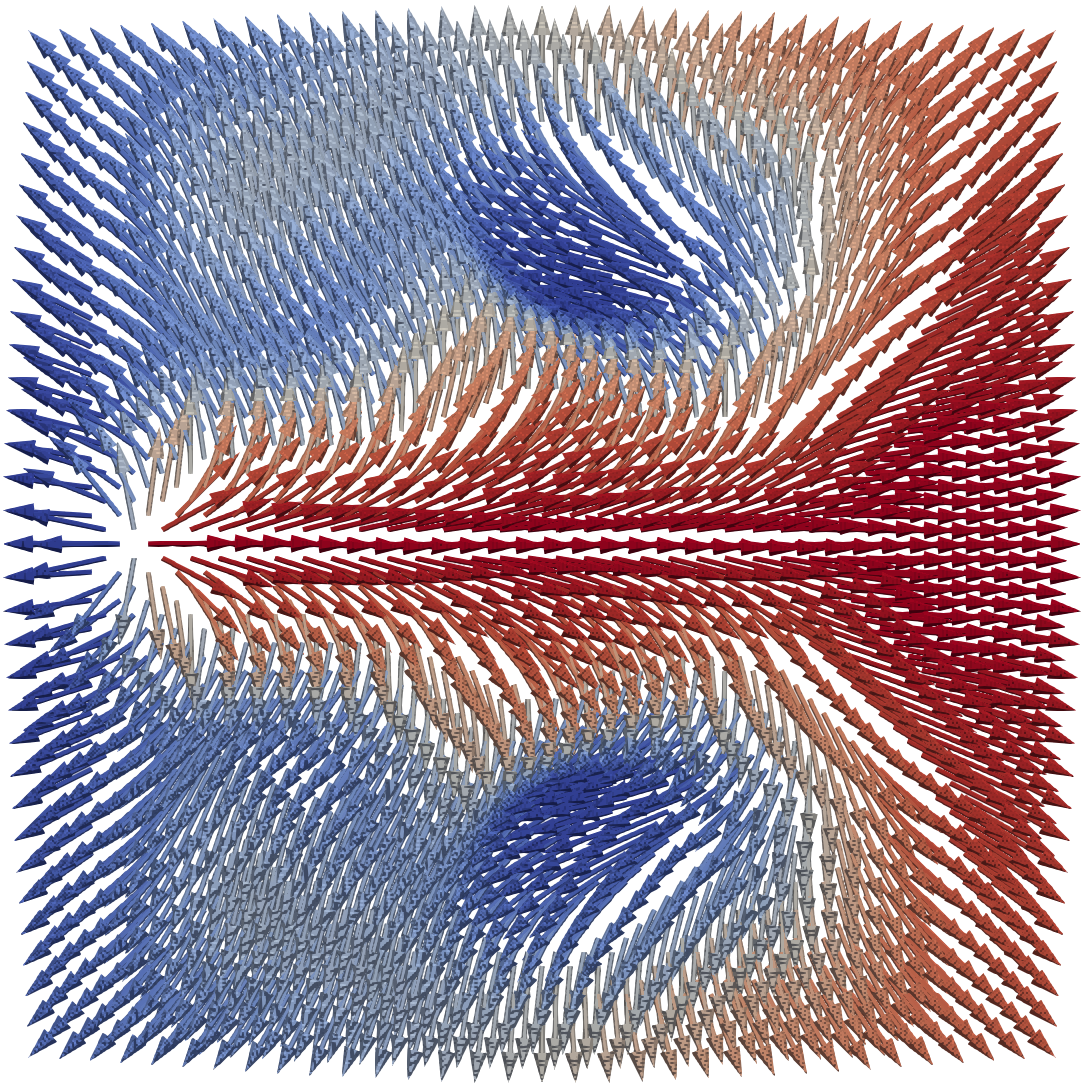}
\caption{Numerical solution computed with $h=2^{-5}$ (without stabilization) at time $t=0,0.005,0.05,0.075,0.09,0.095$ (from left to right, top to bottom).}
\label{fig_uh_stab2}
\end{figure}

}

\bibliographystyle{plain}
\bibliography{refs}

\end{document}